\documentclass[12pt]{amsart}

\usepackage{amsfonts,amssymb}

\newtheorem{theorem}{Theorem}[section]
\newtheorem*{theorem*}{Theorem}
\newtheorem{lemma}[theorem]{Lemma}
\newtheorem{corollary}[theorem]{Corollary}
\newtheorem*{corollary*}{Corollary}
\newtheorem{proposition}[theorem]{Proposition}
\newtheorem*{proposition*}{Proposition}

\newtheorem{question}[theorem]{Question}

\theoremstyle{definition}
\newtheorem*{definition*}{Definition}
\newtheorem{definition}[theorem]{Definition}

\theoremstyle{remark}

\newtheorem*{remark*}{Remark}

\theoremstyle{plain}

\DeclareMathOperator{\Av}{Av}
\DeclareMathOperator{\iter}{Iter}
\DeclareMathOperator*{\llimsup}{limsup}
\renewcommand{\limsup}{\llimsup}
\DeclareMathOperator{\Nil}{Nil}
\DeclareMathOperator{\Bohr}{Bohr}
\DeclareMathOperator{\PW}{PW-}

\DeclareMathOperator{\SH}{SG}

\DeclareMathOperator{\Sumset}{S}
\newcommand{\bbullet}{{e,e}}

\newcommand{\wh}{\widehat}
\newcommand{\norm}[1]{\Vert #1 \Vert}
\newcommand{\nnorm}[1]{|\!|\!| #1 |\!|\!|}
\newcommand{\type}[1]{^{[#1]}}
\newcommand{\cube}[1]{^{(#1)}}

\newcommand{\Z}{{\mathbb Z}}
\newcommand{\T}{{\mathbb T}}

\newcommand{\N}{{\mathbb N}}
\newcommand{\E}{{\mathbb E}}
\newcommand{\CA}{{\mathcal A}}
\newcommand{\CC}{{\mathcal C}}

\newcommand{\CF}{{\mathcal F}}
\newcommand{\CB}{{\mathcal B}}

\newcommand{\CP}{{\mathcal P}}

\newcommand{\bg}{\mathbf{g}}
\newcommand{\bh}{\mathbf{h}}

\newcommand{\bx}{\mathbf{x}}
\newcommand{\by}{\mathbf{y}}

\newcommand{\bI}{\mathbf{I}}

\newcommand{\one}{{\boldsymbol 1}}

\begin{document}
\title{Nil-Bohr Sets of Integers}

\author{Bernard Host}
\address{Laboratoire d'analyse et de math\'ematiques appliqu\'{e}es, Universit\'e 
de Marne la Vall\'ee \& CNRS UMR 8050\\
5 Bd. Descartes, Champs sur Marne\\
77454 Marne la Vall\'ee Cedex 2, France}
\email{bernard.host@univ-mlv.fr}

\author{Bryna Kra}
\address{Department of Mathematics, Northwestern University \\ 2033 Sheridan Road Evanston \\ IL 60208-2730, USA} 
\email{kra@math.northwestern.edu}

\thanks{The first author was partially supported by the Institut
Universitaire de France and the second by NSF grant 
$0555250$ and by the Clay Mathematics Institute. This work was begun 
during the visit of the authors to MSRI and completed while the 
second author was a visitor at Institut Henri Poincar\'e; 
we thank the institutes
for their hospitality.}

\begin{abstract}
We study relations between subsets of integers that are large, 
where large can be interpreted in terms of size (such as a set of 
positive upper density or a set with bounded gaps) or in terms of 
additive structure (such as a Bohr set).  Bohr sets are fundamentally 
abelian in nature and are linked to Fourier analysis.  Recently it 
has become apparent that a higher order, non-abelian, Fourier 
analysis plays a role in both additive combinatorics and in ergodic 
theory.  Here we introduce a higher order version of Bohr 
sets and give various properties of these objects, generalizing 
results of Bergelson, Furstenberg, and Weiss.
\end{abstract}
\maketitle

\section{Introduction}

\subsection{Additive combinatorics and Bohr sets}
Additive combinatorics is the study of structured subsets 
of integers, concerned with questions such as what 
one can say about sets of integers 
that are 
large in terms of size or about sets that are large in terms of 
additive structure.  An interesting problem 
is finding various relations between classes of large sets.

Sets with positive upper Banach density or syndetic sets\footnote{If $A\subset\Z$, the 
{\em upper Banach density} $d^*(A)$ is defined to be $$\limsup_{b_n-a_n\to\infty}\frac{|A\cap 
[a_n,b_n]|}{b_n-a_n} \ .$$  The set $A\subset\Z$ is said to be {\em 
syndetic} if it intersects every sufficiently large interval.} 
are examples of sets that are large in terms of size.  
A simple result relates these two notions: if $A\subset\Z$ has positive 
upper Banach density, then the set of differences $\Delta(A) = A-A = \{a-b\colon a,b\in 
A\}$ is syndetic. 

An example of a structured set is a Bohr set.  
Following a modification of the traditional definition
introduced in~\cite{BFW}, we say that 
a subset $A\subseteq\Z$ is a $\Bohr$ set if there exist 
$m\in\N$, $\alpha\in\T^m$, and an open set $U\subset\T^m$ such that 
$$
\{n\in\Z\colon n\alpha\in U\}
$$
is contained in $A$ (see 
Definition~\ref{def:Bohr}).
It is easy to check that the class of Bohr sets is closed under translations. 

Most of the notions of a large set that are defined solely in terms of size are 
also closed under translation.  However, we have important classes of structure sets 
that are not closed under translation.  One particular example is that of 
a $\Bohr_0$-set: a subset $A\subseteq\Z$ is a $\Bohr_0$-set if  it is a Bohr set such 
that the set $U$ in the previous definition contains
$0$.

A simple application of the pigeonhole principle gives that if $S$ 
is an infinite set of integers, then $S-S$ has nontrivial 
intersection with every $\Bohr_0$-set.  
This is another example of largeness: a set is large if it has 
nontrivial intersection with every member of some class of sets.  
Such notions of largeness are generally referred to as dual notions and are
denoted with a star.  For example, a $\Delta^*$-set is a set that has 
nontrivial intersection with the set of differences $\Delta(A)$ from any infinite set $A$.  

Here we study converse results.  If a set intersects every set of a 
given class, then our goal is to show that it has some sort of structure. 
Such theorems are not, in general, exact converses of the direct structural 
statements.  For example, there exist $\Delta^*$-sets that are not 
$\Bohr_0$-sets (see~\cite{BFW}).  But, this statement is not far 
from being true.  Strengthening a result of~\cite{BFW}, we show 
(Theorem~\ref{th:Delta}) that
a $\Delta^*$-set is a piecewise $\Bohr_0$-set, 
meaning that it agrees with a $\Bohr_0$-set on a sequence of 
intervals whose lengths tend to infinity.

\subsection{Nil-Bohr sets}
Bohr sets are fundamentally linked to abelian groups and Fourier 
analysis.  In the past few years, it has become apparent in both 
ergodic theory and additive combinatorics that nilpotent groups and 
a higher order Fourier analysis play a role (see, for example,~\cite{gowers},~\cite{HK}, and~\cite{GT}). 
As such, we define a $d$-step nil-$\Bohr_0$-set, analogous to the definition of a 
$\Bohr_0$-set, but with a nilmanifold replacing the role of an abelian 
group (see Definition~\ref{def:nilBohr}).  
For $d=1$, the abelian case, this is exactly the object studied 
in~\cite{BFW}.  Here we generalize their results for $d\geq 1$.

We obtain a generalization of Theorem~\ref{th:Delta} on different sets, introducing 
the idea of a set of sums with gaps.  For an integer $d\geq 0$ and
an infinite sequence $P=(p_i)$ in  $\N$, the
\emph{set of sums with  gaps of length $< d$} of $P$ 
is defined to be the set $\SH_d(P)$ 
of all integers of the form
$$
\epsilon_1p_1+\epsilon_{2}p_{2}+\dots+\epsilon_np_n\ ,
$$
where   $n\geq 1$ is an integer, $\epsilon_i\in\{0,1\}$  for $1\leq i\leq n$, 
the $\epsilon_i$ are not all equal to $0$, and the blocks of 

Our main result (Theorem~\ref{th:sums})
is that a set with nontrivial intersection with 
any $\SH_d$-set is a piecewise $d$-step nilpotent $\Bohr_0$-set.  

\subsection{The method}
The first ingredient in the proof is a modification and extension of 
the Furstenberg Correspondence Principle.  The classical Correspondence 
Principle gives a relation between sets of integers and measure 
preserving systems, relating the size of the sets of integers to the
measure of some sets of the system.  It does not 
give relations between structures in the set of integers under 
consideration and ergodic properties of the corresponding system.  
Some information of this type is provided by our modification 
(originally introduced in~\cite{HK3}).  

We then are left with studying certain properties of the systems that arise 
from this correspondence.  As in several related problems, the properties of the 
system that we need are linked to certain factors of the system, 
which are nilsystems.  This method and these factors were introduced 
in the study of convergence of some multiple ergodic averages 
in~\cite{HK}.  

Working within these factor systems, we conclude 
by making use of techniques for the analysis of nilsystems that have been developed 
over the last few years.  
In the abelian setting, a fundamental tool is the Fourier transform, 
but no analog exists for higher order nilsystems\footnote{The theory 
of representations does not help us, as the interesting 
representations of a nilpotent Lie group are infinite dimensional.}.  
Another classical tool available in the abelian case is the 
convolution product, but this too is not defined for general 
nilsystems.   Instead, in Section~\ref{sec:four} we build some spaces
and measures that take on the role of the convolution.  
As an example, if $G$ is a compact abelian group 
we can consider the subgroup 
$$\{(g_1, g_2, g_3, g_4)\in G^4\colon g_1+g_2 
= g_3+g_4\}$$ 
of $G^4$, and we take integrals with respect to its Haar 
measure.  This replaces the role of the convolution product.  

These constructions are then used to prove the key convergence result 
(Proposition~\ref{prop:ouf}).  By studying the limit, 
Theorem~\ref{th:sums} is deduced in Section~\ref{sec:proofmain}.  
By further iterations, Theorems~\ref{th:Delta} and~\ref{th:holes} are 
proved in Section~\ref{sec:iteration}.

\subsection*{Acknowledgment}  Hillel Furstenberg introduced us to this 
problem and we thank him for his encouragement, as well as for helpful comments on 
a preliminary version of this article. 

\section{Precise statements of definitions and results}

\subsection{Bohr sets and Nil-Bohr sets} 

We formally define the objects described in the introduction:
\begin{definition}
\label{def:Bohr}
A subset $A\subseteq\Z$ is a $\Bohr$ set if there exist 
$m\in\N$, $\alpha\in\T^m$, and an open set $U\subset\T^m$ such that 
$$
\{n\in\Z\colon n\alpha\in U\}
$$
is contained in $A$; the set $A$ is a $\Bohr_0$-set if additionally $0\in U$.
\end{definition}

Note that these sets can also be defined in terms of the topology induced on $\Z$ by embedding 
the integers into the Bohr compactification: a subset of $\Z$ is Bohr if it contains an open set 
in the induced topology and is $\Bohr_0$ if it contains an open neighborhood of $0$ 
in the induced topology.

We can generalize the definition of a $\Bohr_0$-set 
for return times in a nilsystem, rather than just in a torus.  
We first give a short definition of a nilsystem and 
refer to Section~\ref{sec:nilsystem} for further properties.  

\begin{definition}
If $G$ is a $d$-step nilpotent Lie group and $\Gamma\subset G$ is a discrete and cocompact 
subgroup, the compact manifold $X = G/\Gamma$ is a {\em $d$-step nilmanifold}.  
The {\em Haar measure} $\mu$ of $X$ is the unique probability measure that is 
invariant under the action $x\mapsto g\cdot x$ of $G$ on $X$ by left translations.  

If $T$ denotes left translation on $X$ by a fixed element of $G$, then $(X, \mu, T)$ 
is a {\em $d$-step nilsystem}.  
\end{definition}

Using neighborhoods of a point, we define a generalization of a Bohr set: 
\begin{definition}
\label{def:nilBohr}
A subset $A\subseteq\Z$ is a $\Nil_d\Bohr_0$-set if there exist a $d$-step 
nilsystem $(X, \mu, T)$, $x_0\in X$, and an open set $U\subset X$ containing $x_0$ such that 
$$
\{n\in\Z\colon T^nx_0\in U\}
$$
is contained in $A$.
 \end{definition}

Similar to the Bohr compactification of $\Z$ that can be used to define the Bohr sets, 
there is a {\em $d$-step nilpotent compactification} of $\Z$ that can be used 
to define the $\Nil_d\Bohr_0$-sets.    
This compactification is a non-metric compact space $\wh{Z}$, 
endowed with a homeomorphism $T$ and a particular point $\wh{x_0}$ 
with dense orbit, 
and is characterized by the following properties:
\begin{enumerate}
\item Given any $d$-step nilsystem $(Z, T)$ and a point $x_0\in Z$, 
there is a unique factor map $\pi_Z\colon \wh{Z}\to Z$ with 
$\pi_Z(\wh{x_0}) = x_0$.

\item  The topology of $\wh Z$ is spanned by these factor maps 
$\pi_Z$.
\end{enumerate}

\begin{remark*}
A $\Bohr_0$-set can be defined in terms of almost periodic 
sequences.  In the same way, a $\Nil_d\Bohr_0$-set can be defined in 
terms of some particular sequences, the $d$-step {\em nilsequences}.  
Since $\Nil_d\Bohr_0$-sets are defined locally, it seems likely that 
they can be defined by certain particular types of nilsequences, 
namely those arising from generalized polynomials without constant 
terms.  We do not address this issue here.
\end{remark*}

\subsection{Piecewise versions}

If $\CF$ denotes a class of subsets of integers, various authors, 
for example Furstenberg in~\cite{Furstenberg} and 
Bergelson, Furstenberg, and Weiss in~\cite{BFW}, define a subset $A$ 
of integers to be a {\em piecewise-$\CF$ set} if $A$ contains the intersection of 
a sequence of arbitrarily long intervals
and a member of $\CF$.  For example, 
the notions of piecewise-Bohr set,
a piecewise-$\Bohr_0$-set, and a 
piecewise-$\Nil_d\Bohr_0$-set, can be defined in this way.

However, the notion of a piecewise set is rather weak: for example, a 
piecewise-Bohr set defined in this manner is not necessarily syndetic. 
The properties that we can prove are stronger than the traditional 
piecewise statements, and in particular imply 
the traditional piecewise versions.  
For this, we introduce a 
stronger definition of piecewise:
\begin{definition}
Given a class $\CF$ of subsets of integers, 
the set $A\subset\Z$ is said to be {\em strongly piecewise}-$\CF$, written $\PW\CF$, 
if for every sequence $(J_k\colon k\geq 1)$ of 
intervals whose lengths $|J_k|$ tend to $\infty$, there exists a 
sequence $(I_j\colon j\geq 1)$ of 
intervals satisfying: 
\begin{enumerate}
\item For each $j\geq 1$, there exists some $k = k(j)$ 
such that the  interval $I_j$ is contained in $J_k$;
\item The lengths $|I_j|$ tend to infinity;
\item There exists a set $\Lambda\in\CF$ such that 
$\Lambda\cap I_j\subset A$ for every $j\geq 1$.
\end{enumerate}
\end{definition}

Note that $\Lambda$ depends on the sequence $(J_k)$.
With this definition of strongly piecewise, if the class $\CF$ consists of 
syndetic sets then every $\PW\CF$-set is syndetic. In particular, 
a strongly piecewise-Bohr set, denoted 
$\PW\Bohr$, is syndetic.  Similarly, we denote a strongly piecewise-$\Bohr_0$-
set by $\PW\Bohr_0$ and a strongly piecewise-$\Nil_d\Bohr_0$-set by 
$\PW\Nil_d\Bohr_0$ and these sets are  also syndetic.

\subsection{Sumsets and Difference Sets}
\begin{definition}
Let $E\subset\N$ be a set of integers. The {\em sumset of $E$} 
is the set $\Sumset(E)$ 
consisting of all nontrivial 
finite sums of distinct elements of $E$.  

A subset $A$ of $\N$ is a {\em $\Sumset_r^*$-set} if $A\cap  
\Sumset(E)\neq\emptyset$ for every set $E\subset\N$ with $|E|=r$.
\end{definition}

We have: 
\begin{theorem}
\label{th:sums}
Every $S_{d+1}^*$-set is a $\PW\Nil_d\Bohr_0$-set.
\end{theorem}

We can iterate this result, leading to the following definitions from~\cite{BFW} 
and~\cite{Furstenberg}:

\begin{definition}
If $S$ is a nonempty subset of $\N$, define the {\em difference set} $\Delta(S)$ by
$$
 \Delta(S)=(S-S)\cap\N=\{b-a\colon a\in S,\ b\in S,\ b>a\}\ .
$$
If $A$ is a subset of $\N$, 
$A$ is a {\em $\Delta^*_r$-set} if $A\cap\Delta(S)\neq\emptyset$ 
for every  subset $S$ of $\N$ with $|S|=r$; 
$A$ is a {\em $\Delta^*$-set} if $A\cap\Delta(S)\neq\emptyset$ for every 
infinite subset $S$ of $\N$.
\end{definition}

\begin{theorem}
\label{th:Delta}
Every $\Delta^*$-set is a $\PW\Bohr_0$-set.
\end{theorem}

Every $\Delta^*_r$ set is obviously a $\Delta^*$-set and 
Theorem~\ref{th:Delta} generalizes a result of~\cite{BFW}. 
The class of sets of the form $\Delta(S)$ with $|S|=3$ coincides with 
the class of sets of the form $\Sumset(E)$ with $|E|=2$ and thus 
the classes $\Delta_3^*$ and $\Sumset_2^*$ are the same.  
Theorem~\ref{th:Delta} generalizes the case $d=1$ of Theorem~\ref{th:sums}.

The converse statement of Theorem~\ref{th:Delta} does not hold.  However, it is easy 
to check that every $\Bohr_0$-set is a $\Delta^*$-set (see~\cite{BFW}).

\begin{definition}
\label{def:SH}
Let $d\geq 0$ be an integer and
let $P=(p_i)$ be a (finite or infinite)  sequence in  $\N$. 
The \emph{set of sums with  gaps of length $< d$} of $P$ 
is the set $\SH_d(P)$ 
of all integers of the form
$$
\epsilon_1p_1+\epsilon_{2}p_{2}+\dots+\epsilon_np_n\ ,
$$
where   $n\geq 1$ is an integer, $\epsilon_i\in\{0,1\}$  for $1\leq i\leq n$, 
the $\epsilon_i$ are not all equal to $0$, and the blocks of 
consecutive $0$'s between two $1$'s have length $<d$.

A subset $A\subseteq\N$ is an $\SH_d^*$-set if $A\cap 
\SH_d(P)\neq\emptyset$ for every infinite sequence $P$ in $\N$.
\end{definition}

Note that in this definition, $P$ is a sequence and not a subset of 
$\N$.  

For example, if $P=\{p_1,p_2,\dots\}$, 
then $\SH_1(P)$ is the set of all sums $p_m+\dots+p_n$ of consecutive 
elements of $P$, 
and thus it coincides with the set $\Delta(S)$ where 
$S=\{s,s+p_1,s+p_1+p_2,\dots\}$.
Therefore $\SH_1^*$-sets are the same as $\Delta^*$-sets.

For a sequence $P$, $\SH_2(P)$ consists of all sums of the form
$$
\sum_{i=m_0}^{m_1}p_i+\sum_{i=m_1+2}^{m_2}p_i+\dots+
\sum_{i=m_{k-1}+2}^{m_k}p_i+\sum_{i=m_k+2}^{m_{k+1}}p_i\ ,
$$
where $k\in\N$ and $m_0, m_1,\ldots, m_{k+1}$ are 
positive integers satisfying $m_{i+1}\geq m_i +2$ for $i=0, \ldots, k$. 

\begin{theorem}
\label{th:holes}
Every $\SH_d^*$-set is  a $\PW\Nil_d\Bohr_0$-set.
\end{theorem}
If $|P|=d+1$, then $\SH_d(P)=\Sumset(P)$ and thus 
Theorem~\ref{th:holes} generalizes Theorem~\ref{th:sums}.

In general, a $\Nil_d\Bohr_0$-set is not a $\Delta^*$-set.  
To construct an example, take an irrational $\alpha$ and let 
$\Omega$ be the set of $n\in\Z$ such that $n^2\alpha$ is 
close to $0\mod 1$.  Then $\Omega$ is a $\Nil_2\Bohr_0$-set, as can be checked by considering the transformation on $\T^2$ defined by 
$(x,y)\mapsto (x+\alpha, y+x)$.  
On the other hand, by induction we can build an increasing 
sequence $n_j$ of integers such that $n_j^2\alpha$ is 
close to $1/3\mod 1$, while $n_in_j\alpha\mod 1$ is close 
to $0$ for $i< j$.  Taking $S$ to be the set of such $n_j$,  
we have that $\Delta(S)$ does not intersect $\Omega$.

This leads to the following question:
\begin{question}
Is every  $\Nil_d\Bohr_0$-set an $\SH_d^*$-set?
\end{question}

As our characterizations of the sets $\SH_d$ and the class $\SH_d^*$ 
are complicated, we ask the following:
\begin{question}
Find an alternate description of the sets $\SH_d$ and of the 
class $\SH_d^*$.
\end{question}

\section{Preliminaries}

\subsection{Notation}
We introduce notation that we use throughout the remainder of the 
article.

If $X$ is a set and $d\geq 1$ is an integer, we write $X\type d = X^{2^d}$ and we index the 
$2^d$ copies of $X$ by $\{0,1\}^d$.  
Elements of $X\type d$ 
are written as 
$$
\bx = (x_\epsilon\colon\epsilon\in\{0,1\}^d) \ .
$$
We write elements of $\{0,1\}^d$ without commas or parentheses.  

We also often identify $\{0,1\}^d$ with the family $\CP([d])$ of 
subsets of $[d] = \{1, 2, \ldots, d\}$.  In this identification, 
$\epsilon_i = 1$ is the same as $i\in\epsilon$ and  $\emptyset = 
00\ldots 0$.

For $\epsilon\in\{0,1\}^d$ and $n\in\Z^d$, we write $|\epsilon| = \epsilon_1+\ldots+\epsilon_d$ 
and $\epsilon\cdot n = \epsilon_1n_1+\ldots+\epsilon_dn_d$.

If $p\colon X\to Y$ is a map, then we write $p\type d\colon X\type 
d\to Y\type d$ for the map $(p, p, \ldots, p)$ taken $2^d$ times.  In 
particular, if $T$ is a transformation on the space $X$, 
we define $T\type d\colon X\type d\to X\type d$ as $T\times T\times\ldots\times T$ taken $2^d$ times.  
We define the {\em face transformations }
$T_i\type d$ for $1\leq i\leq d$ by:
$$
(T_i\type d \bx)_\epsilon = 
\begin{cases}
T(x_\epsilon) & \text{ if } \epsilon_i=1 \\
x_\epsilon & \text{ otherwise \ .}
\end{cases}
$$

In a slight abuse of notation, we denote all transformations, even in 
different systems, by the letter $T$ (unless the system is naturally 
a Cartesian product).

For convenience, we assume that all functions are real 
valued.  

\subsection{Review of nilsystems}
\label{sec:nilsystem}

\begin{definition}
If $G$ is a $d$-step nilpotent Lie group and $\Gamma\subset G$ is a discrete and cocompact 
subgroup, the compact manifold $X = G/\Gamma$ is a {\em $d$-step nilmanifold}.  
The {\em Haar measure} $\mu$ of $X$ is the unique probability measure that is 
invariant under the action $x\mapsto g\cdot x$ of $G$ on $X$ by left translations.  

If $T$ denotes left translation on $X$ by a fixed element of $G$, then $(X, \mu, T)$ 
is a {\em $d$-step nilsystem}.  
\end{definition}

(We generally omit the $\sigma$-algebra from the notation, writing $(X, \mu, T)$ 
for a measure preserving system rather than $(X, \CB, \mu, T)$, where $\CB$ 
denotes the Borel $\sigma$-algebra.)

A $d$-step nilsystem is an example of a topological 
distal dynamical system.  For a $d$-step nilsystem, 
the following properties are equivalent: transitivity, minimality, 
unique ergodicity, and ergodicity.  (Note that the first 
three of these properties refer to the topological system, 
while the last refers to the measure preserving system.)
Also, the closed orbit of a point in a
$d$-step nilsystem is a $d$-step nilsystem, 
and it follows that this closed orbit is minimal and uniquely ergodic.
See~\cite{AGH} for proofs
and general references on nilsystems.

We also speak of a nilsystem $(X = G/\Gamma, T_1, \ldots, T_d)$, where $T_1, 
\ldots, T_d$ are translations by commuting elements of $G$.  
All the above properties hold for such systems. 

We also make use of inverse limits of systems, both in the 
topological and measure theoretic senses.  All inverse limits 
are implicitly assumed to be taken along sequences.  
Inverse limits for a sequence of nilsystems are the same in both the 
topological and measure theoretic senses: this follows because a 
measure theoretic factor map between two nilsystems is necessarily 
continuous.  

Many properties of the nilsystems also pass to the inverse limit.
In particular, in an inverse limit of $d$-step nilsystems, every 
closed orbit is minimal and uniquely ergodic.

\subsection{Structure Theorem}
Assume now that $(X, \mu, T)$ is an ergodic system.  

We recall a construction and definitions from~\cite{HK}, but for 
consistency we make some small changes in the notation.  
For an integer $d\geq 0$, 
a measure $\mu\type d$ on $X\type d$ was built in~\cite{HK}.  Here we 
denote this measure by $\mu\cube d$. 

The measure $\mu\cube d$ 
is invariant under $T\type d$ and under all the face 
transformations $T\type d_i$, $1\leq i\leq d$.  
Each of the projections of the measure $\mu\cube d$ on $X$ is equal 
to the measure $\mu$.  

If $f$ is a bounded measurable function on $X$, then 
$$
\int \prod_{\epsilon\subset [d]}f(x_\epsilon)\,d\mu\cube d(\bx) \geq 
0 $$
and we define
$\nnorm f_d$ to be this expression raised to the power $1/2^d$.  
Then $\nnorm \cdot_d$ is a seminorm on $L^\infty(\mu)$.  A main 
result from~\cite{HK} is that this is a norm if and only if the 
system is an inverse limit of $(d-1)$-step nilsystems.  
More precisely, a summary of the Structure Theorem of~\cite{HK} is:
\begin{theorem}
Assume that $(X, \mu, T)$ is an ergodic system.  Then for each $d\geq 2$, there exist a 
system $(Z_d, \mu_d, T)$ and a factor map $\pi_d\colon X\to Z_d$ satisfying:
\begin{enumerate}
\item $(Z_d, \mu_d, T)$ is the inverse limit of a sequence of $(d-1)$-step nilsystems.
\item For each $f\in L^\infty(\mu)$, $\nnorm{f-\E(f\mid Z_d)\circ\pi_d}_d = 0$.  
\end{enumerate}
\end{theorem}

For each $d\geq 1$, we call  $(Z_d,\mu_d,T)$ the {\em HK-factor of order $d$} 
of $(X,\mu,T)$. 
The factor map $\pi_d\colon X\to Z_d$ is measurable, and a priori 
has no reason to be continuous.
For $\ell\leq d$, $Z_\ell$ is a factor of $Z_d$, with a continuous factor map.

If $(X, \mu, T)$ is an ergodic inverse limit of $(d-1)$-step 
nilsystems, we define\footnote{We are forced to use different 
notation from that in~\cite{HK}, as otherwise the proliferation of 
indices would be uncontrollable.} $X\cube d$ to the closed orbit in $X\type d$ of 
a point $\bx_0=(x_0, \ldots, x_0)$ (for some arbitrary 
$x_0\in X$) under the transformations $T\type d$ and $T\type 
d_i$ for $1\leq i\leq d$.  The system $X\cube d$, endowed with these 
transformations, is minimal and uniquely ergodic.  Its unique 
invariant measure is exactly the measure $\mu\cube d$ described 
above.  When $(X, \mu, T)$ is a $(d-1)$-step nilsystem, then $X\cube 
d$ is a nilmanifold and $\mu\cube d$ is its Haar measure.  

\subsection{Furstenberg correspondence principle revisited}
\label{subsec:Furstenberg}
By $\ell^\infty(\Z)$, we mean the algebra of bounded real valued 
sequences indexed by $\Z$.
 
Let $\CA$ be a subalgebra of 
$\ell^\infty(\Z)$, containing the constants,
 invariant under the shift, 
closed and separable with respect to the norm $\norm\cdot_\infty$ of 
uniform convergence.   We refer to this simple as  ``an algebra.''
In applications, finitely many subsets  of $\Z$ are given and 
$\CA$ is the shift invariant algebra spanned by indicator 
functions of these subsets.  

Given an algebra, we associate various objects to it: 
a dynamical system, 
an ergodic measure on this system, a sequence of intervals, etc.  
We give a summary of these objects without proof, referring 
to~\cite{HK3} for more information.

\subsubsection{A system associated to $\CA$}
There exist a topological dynamical system $(X,T)$ and  a point $x_0\in X$ such that the map 
$$
 \phi\in\CC(X)\mapsto \bigl(\phi(T^nx_0)
\colon n\in\Z\bigr)\in\ell^\infty(\Z)
$$
is an  isometric isomorphism of algebras from $\CC(X)$ onto $\CA$.
(We use $\CC(X)$ to denote the collection of continuous functions on 
$X$.)

In particular, if $S$ is a subset of $\Z$ with $\one_S\in\CA$, then 
there exists a subset $\widetilde S$ of $X$ that is open and 
closed in $X$ such that 
\begin{equation}
\label{eq:StildeS}
 \text{for every }n\in\Z,\quad T^nx_0\in\widetilde S
\text{ if and only if }n\in S\ .
\end{equation}

\subsubsection{Some averages and some measures 
associated to $\CA$}
\label{subsec:intervals}
There also exist a sequence $\bI=(I_j\colon j\geq 1)$ of intervals of 
$\Z$, whose lengths tend to infinity, and an invariant ergodic 
probability measure $\mu$ on $X$ such that
\begin{equation}
\label{eq:intphi}
\text{for every }\phi\in\CC(X),\quad
\frac 1{|I_j|}\sum_{n\in I_j}\phi(T^nx_0)\to\int \phi\,d\mu\text{ as }
j\to+\infty\ .
\end{equation}

Given a subset $S$ of $\Z$, we can chose the intervals $I_j$ such that
$$
 \frac{|S\cap I_j|}{|I_j|}\to d^*(S)\text{ as }j\to+\infty\ ,
$$
where $d^*(S)$ denotes the upper Banach density of $S$.

In particular, we can assume that the 
intervals $I_j$ are contained in $\N$.  

\subsubsection{Notation}
In the sequel, when $a=(a_n\colon n\in\Z)$ is a bounded sequence,
we write
$$
 \lim\Av_{n,\bI} a_n=\lim_{j\to+\infty}\frac 1{|I_j|}
\sum_{n\in I_j} a_n
$$
if this limit exists, and set
$$
 \limsup \bigl| \Av_{n,\bI} a_n\bigr|
=\limsup_{j\to+\infty}\bigl|\frac 1{|I_j|}
\sum_{n\in I_j} a_n\bigr|\ .
$$
We omit the subscripts $n$ and/or $\bI$ if they are clear from the context.
\subsubsection{Averages and factors of order $k$}
\label{subsec:averagesZk}

Recall that $Z_k$ denotes the HK-factor of order $k$ of $(X, \mu, T)$ and 
that $\pi_k\colon X\to Z_k$ denotes the factor 
map. 

The sequence of intervals $\bI=(I_j\colon j\geq 1)$ can be chosen 
such that:
\begin{proposition}
\label{prop:intphif}
For every $k\geq 1$, there exists a point $e_k\in Z_k$
such that $\pi_{\ell,k}(e_k)=e_\ell$ for $\ell <k$ and such that
for every $\phi\in\CC(X)$ 
 and every $f\in\CC(Z_k)$, 
 $$
\lim \Av_{\bI}\phi(T^nx_0)f(T^ne_k)=\int \phi\cdot 
f\circ\pi_k\,d\mu=\int \E(\phi\mid Z_k)\,f\,d\mu_k\ .
$$
\end{proposition}
This formula extends~\eqref{eq:intphi}.

The next  corollary is an example of the relation between integrals on the factors $Z_k$ and 
$\PW\Nil\Bohr$-sets. More precise results are proved and used in 
the sequel.
\begin{corollary}
\label{cor:intphif}
Let $S$ be a subset of $\Z$ such that $\one_S$ belongs to the algebra 
$\CA$ and let $\widetilde S$ be the corresponding subset of $X$. 
Let $f$ be a nonnegative continuous function on $Z_k$ with 
$f(e_k)>0$, where $e_k$ is as in Proposition~\ref{prop:intphif}. If
$$
 \int \one_{\widetilde S}(x)\cdot f\circ\pi_k(x)\,d\mu(x)=0
$$
then $\Z\setminus S$ is a $\PW\Nil_k\Bohr_0$-set.
\end{corollary}
\begin{proof}
Let $\Lambda=\{n\in\Z\colon f(T^ne_k)>f(e_k)/2\}$. Then $\Lambda$ is a 
$\Nil_k\Bohr_0$-set. 
By Proposition~\ref{prop:intphif} and definition~\eqref{eq:StildeS}
 of $\widetilde S$, 
the averages on $I_j$ of 
$1_{S}(n)f(T^ne_k)$ converge to zero.  Thus
$$
\lim_{j\to+\infty} \frac{|I_j\cap S\cap\Lambda|}{|I_j|} =0\ .
$$
Therefore, the subset $E=\bigcup_jI_j\setminus (S\cap\Lambda)$ contains 
arbitrarily long intervals $J_\ell$, $\ell\geq 1$. For every $\ell$,
$J_\ell\cap (\Z\setminus S)\supset J_\ell\cap\Lambda$.
\end{proof}

\subsubsection{}
It is easy to check that given
a sequence of intervals $(J_k\colon k\geq 1)$ whose lengths tend to 
infinity,  we can choose the intervals $(I_j\colon j\geq 1)$ satisfying all of the above 
properties, and such that each interval $I_j$ is a 
subinterval of some $J_k$. To see this, we first reduce to the case 
that the intervals $J_k$ are disjoint and separated by sufficiently 
large gaps. We set $S$ to be the union of these intervals. We have 
$d^*(S)=1$ and we can choose intervals $I'_j$  with $|S\cap 
I'_j|/|I'_j|\to 1$. For 
every $j\in\N$, there exists $k_j$ such that $|I'_j\cap J_{k_j}|/|I'_j|\to 1$ 
as $j\to+\infty$. We set $I_j=I'_j\cap J_{k_j}$ and the sequence 
$(I_j\colon j\geq 1)$ satisfies all the requested properties.

\subsection{Definition of the uniformity seminorms} 
\label{subsec:uniformity}
We recall definitions and results of~\cite{HK3} adapted to the 
present context.  We keep notation as in the previous sections; in 
particular, $Z_k$ and $e_k$ are as in Proposition~\ref{prop:intphif}.

Let $\bI$ 
be as in Section~\ref{subsec:Furstenberg} and let $\CB$  be the algebra spanned by $\CA$ 
and sequences of the form $(f(T^ne_k)\colon n\in\Z)$, where $f$ 
is a continuous function on $Z_k$ for some $k$.  
By Proposition~\ref{prop:intphif}, for every sequence $a = 
(a_n\colon n\in\Z)$ belonging to the algebra $\CB$, the limit $\lim\Av_{\bI, 
n}a_n$ exists.  

Given a sequence $a\in\CB$, for $h = (h_1, \ldots, h_d)\in\Z^d$, 
let
$$
c_\bh = \lim\Av_{\bI,n}\prod_{\epsilon\subset[d]}a_{n+\epsilon\cdot \bh} 
\ .
$$
Then 
$$
\lim_{H\to\infty}\frac{1}{H^d}\sum_{h_1, \ldots, h_d=0}^{H-1}c_\bh
$$
exists and is nonnegative.  We define $\norm a_{\bI, d}$ 
to be this limit raised to the power $1/2^d$.

\begin{proposition}
\label{prop:boundphif}
Let $(Z,T)$ be an inverse limit of  $k$-step nilsystems and 
$f$ be a continuous function on 
$Z$. 
Then for every $\delta>0$ there exists $C=C(\delta)>0$ such that
for every sequence $a=(a_n\colon n\in\Z)$ belonging to the algebra 
$\CB$ and for every $z\in Z$,
$$
\limsup \bigl|\Av_\bI a_n f(T^nz)\bigr|
\leq\delta\norm a_\infty+C\norm a_{\bI,k+1}\ .
$$
\end{proposition}

\begin{proof}
By density, we can reduce to the case that $(Z,T)$ is a $k$-step 
nilsystem and that the function $f$ is smooth.

In this case, the result is contained in~\cite{HK3}
under the hypothesis that the 
system is ergodic.
Indeed, by Proposition~5.6 of this paper, $f$ is a ``dual function'' on $X$.
By the ``Modified Direct Theorem'' of Section~5.4 in~\cite{HK3}, there exists a 
constant $\nnorm f_k^*\geq 0$ with 
$$
 \limsup \bigl|\Av_\bI a_n f(T^nz)\bigr|\leq 
\nnorm f_k^*\cdot \norm a_{\bI,k+1}\ .
$$
In the proofs of~\cite{HK3} we can check that the 
hypothesis of ergodicity is not used.
\end{proof}

The next proposition was proved in~\cite{HK3} and follows from the 
Structure Theorem.
\begin{proposition}
\label{prop:approx}
Let $\phi$ be a continuous function on $X$ with $|\phi|\leq 1$,
$k\geq 1$ an integer, and 
$f$ a continuous function on $Z_k$ with $|f|\leq 1$. 
Then 
$$
\bigl\Vert\bigl(\phi(T^nx_0)-f(T^ne_k)\colon 
n\in\Z\bigr)\bigr\Vert_{\bI,k+1}
\leq 2 \bigl\Vert f-\E(\phi\mid Z_k)\bigr\Vert_{L^1(\mu_k)}^
{1/2^{k+1}}\ .
$$
\end{proposition}

\section{Some measures associated to inverse limits of nilsystems}
\label{sec:four}
\subsection{Standing assumptions}
We assume that 
every topological system $(Z,T)$ is implicitly 
endowed with a particular point, called the \emph{base point}.
Every topological factor map is implicitly assumed to map = base 
point to base point. For every $k\geq 1$, we take the base point 
of $Z_k$ to be the point $e_k$ 
introduced in Section~\ref{subsec:averagesZk}.  

If $(Z,T)$ is a nilsystem with $Z=G/\Gamma$, then by changing the 
group $\Gamma$ if needed,  we can assume that the base 
point of $Z$ is the image in $Z$ of the unit element of $G$.

\subsection{The measures $\mu\cube m_e$}
\begin{proposition}
\label{prop:mesuresmua}
Let $(X,\mu,T)$ be an ergodic inverse limit of ergodic $k$-step 
nilsystems, endowed with the base point $e\in X$, and let $m\geq 1$ be an integer.

\noindent\emph{a)} The closed orbit of the point 
$e\type{m}=(e, e, \ldots, e)$ of $X\cube{m}$ under the transformations $T\type{m}_i$, 
$1\leq i\leq m$, is
$$
X\cube{m}_e=\{\bx\in X\cube{m}\colon x_\emptyset=e\}\ .
$$
\noindent\emph{b)} Let   $\mu\cube{m}_e$ be the unique measure on this set 
invariant under 
these transformations.  Then the image
of $\mu\cube{m}_e$ under each of the natural projections $\bx\mapsto 
x_\epsilon\colon X\type{m}\to X$, 
$\emptyset\neq\epsilon\subset[d]$, is equal to $\mu$.

\noindent\emph{c)} Let $(Y, \nu,T)$ be an inverse limit of $k$-step 
nilsystems and  let
$p\colon X\to Y$ be a  factor map. Then
$\nu\cube{m}_e$ is the image of $\mu\cube{m}_e$ under 
$p\type{m}\colon X\type{m}\to Y\type{m}$. 

\noindent\emph{d)} Let $(Y, \nu,T)$ be the $(m-1)$-step factor of $X$ and
$p\colon X\to Y$ be the factor map.  
Then the measure 
$\mu\cube{m}_e$ is relatively independent with respect to 
$\nu\cube{m}_e$, meaning that when $f_\epsilon$, 
$\emptyset\neq\epsilon\subset[d]$, are $2^{m}-1$ bounded measurable 
functions on $X$,
$$
 \int\prod_{\emptyset\neq\epsilon\subset[d]}f_\epsilon(x_\epsilon)
\,d\mu\cube{m}_e(\bx)
= \int\prod_{\emptyset\neq\epsilon\subset[d]}\E(f_\epsilon\mid Y)(y_\epsilon)
\,d\nu\cube{m}_e(\by)\ .
$$
(The existence of these integrals follows from \emph{b)}.)
\end{proposition}
Following our convention, we assume in c) and d) that $Y$ is endowed 
with a base point and that $p$ maps the base point to the base point. 

\begin{proof}
We first prove a) and d) assuming that $X$ 
is a nilsystem.  
(While the proof is contained in~\cite{HK3}, we sketch it here in order 
to introduce some objects and some notation.)

For $g\in G$ and $F\subset\CP([m])$, we write $g^F$ 
for the element of $G\type m$ given by
$$
 \text{for every }\epsilon\subset[m],\quad
\bigl(g^F)_\epsilon=\begin{cases}
g &\text{if }\epsilon\in F\ ;\\
1 &\text{otherwise.}
\end{cases}
$$

We write $X=G/\Gamma$ and let $\tau$ be the element of $G$ defining the 
transformation $T$ of $X$. 
We can assume  that the base point $e$ of $X$ is the image in $X$ of the 
unit element of $G$.
Since $(X,\mu,T)$ is ergodic, we can also assume that $G$ is spanned by the
connected component of the identity and $\tau$. 
We recall a convenient presentation of $G\cube m$ (see~\cite{HK3} 
and~\cite{GT}).

Let $\alpha_1,\dots,\alpha_{2^m}$ be an enumeration of all 
subsets of $[d]$ such that $|\alpha_i|$ is increasing. 
In particular, $\alpha_1=\emptyset$. 
For $1\leq i\leq 2^m$, let 
$F_i=\{\epsilon\colon \alpha_i\subset\epsilon\subset[m]\}$. 
For every $i$,  $F_i$ is an \emph{upper face} of the cube $\CP([m])$, 
meaning a face containing the vertex $[d]$; its \emph{codimension} is
$|\alpha_i|$. Then 
$F_1,\dots,F_{2^m}$ is an enumeration of all the upper faces, in 
decreasing order of codimension.
In particular, $F_1$ is the whole cube $\CP([m])$.

Each element of $G\cube m$ can be written in a unique 
way as
\begin{equation}
\label{eq:deffacegroup}
 \bh=g_1^{F_1}g_2^{F_2}\dots g_{2^d}^{F_{2^d}}\ ,\text{ where }g_i\in 
G_{|\alpha_i|}\text{ for every }i\ .
\end{equation}
(By convention, $G_0=G$.)

We define
$$
G\cube m_e =\bigl\{\bg\in G\cube m\colon g_\emptyset=1\bigr\}\  .
$$
This  group is closed and normal in $G\cube m$ and every element of 
$G\cube m$ can be written in a unique way as $h\type m\bg$ with 
$h\in G$ and $\bg\in G\cube m_e$. 
Moreover,  $G\cube m_e$ is the set of elements of $G\cube m$ that are
 written as in~\eqref{eq:deffacegroup} 
 with $g_{1}=1$. From this, it is easy to deduce that the commutator 
subgroup of this group is equal to $G\cube m_e\cap (G_2)\type m$.

Clearly, the subset $X\cube m_e$ of $X\cube m$ is invariant 
under $G\cube m_e$ and it follows from the preceding 
description
that the action of this group on this set is transitive.
Therefore, the subgroup
$$
 \Gamma\cube m_e :=\Gamma\cube m\cap G\cube m_e
$$
of $G\cube m_e$ is cocompact in $G\cube m_e$ and we can 
identify $X\cube m_e=G\cube m_e/\Gamma\cube m_e$.

It is easy to check that $G\cube m_e$ is spanned by the 
connected component of its identity and the elements $\tau\type m_i$, 
$1\leq i\leq m$. Moreover, by using the above description of $G\cube 
m_e$, it is not difficult to check that the action induced by 
$T\type m_i$, $1\leq i\leq m$ on the compact abelian group 
$G\cube m_e/(G\cube m_e)_2\Gamma\cube m_e$ is 
ergodic. By a classical criteria~\cite{leibman}, the action of the transformations 
$T\type m_i$ on $X\cube m_e$ is ergodic and thus minimal.
In particular, $X\cube m_e$ is the closed orbit of the point 
$e\type m$ under these transformations. This proves a).

We now prove d). 
The $(m-1)$-step nilfactor $(Y,\nu,T)$ of $(X=G/\Gamma,\mu,T)$ is 
$Y=G/\Gamma G_m$ endowed with its Haar measure.

For every $\epsilon\subset[m]$ with 
$\epsilon\neq\emptyset$ and every $w\in G_m$, we have $w^{\{\epsilon\} 
}\in G\cube m_e$ and thus the Haar measure $\mu\cube m_e$ 
of $X\cube m_e$ is invariant under translation by this element. 
The result follows.

We now turn to the proof of the proposition in the general case.  
For a), the generalization to inverse limits is immediate.

\noindent b) Let $\epsilon\in[d]$ with $\epsilon\neq\emptyset$. 
Let $i\in\epsilon$. Then for every $\bx\in X\cube{m}$ we have 
$Tx_\epsilon=(T\type{m}_i\bx)_\epsilon$. Since the measure 
$\mu\cube{m}_e$ is invariant under $T\cube{m}_i$, its image under 
the projection $\bx\mapsto x_\epsilon$ is invariant under $T$ and 
thus is equal to $\mu$.

Property c) is immediate. 

\noindent d) Let the functions $f_\epsilon$ be as in the 
statement; without loss we can assume that $|f_\epsilon|\leq 1$ for 
every $\epsilon$.

Let $(X_i,\mu_i,T_i)$, $i\geq 1$, be an increasing sequence of  $k$-step 
nilsystems with inverse limit $(X,\mu,T)$ and let $\pi_i\colon X\to 
X_i$, $i\geq 1$, be the (pointed) factor maps.

For every $\epsilon$,  $\emptyset\neq\epsilon\subset[d]$, we have that
$$
 \bigl\Vert f_\epsilon-\E(f_\epsilon\circ 
X_i)\circ\pi_i\bigr\Vert_{L^1(\mu)}\to 0 \text{ as }i\to+\infty
$$
and thus
\begin{equation}
\label{eq:XiX}
 \int\prod_{\emptyset\neq\epsilon\subset[d]}
\E(f_\epsilon\mid X_i)\circ\pi_i (x_\epsilon)
\,d\mu\cube{m}_e(\bx)
\to
\int\prod_{\emptyset\neq\epsilon\subset[d]}f_\epsilon(x_\epsilon)
\,d\mu\cube{m}_e(\bx)
\end{equation}
 as $i\to+\infty$.

For every $i$, let  $(W_i,\sigma_i,T)$ be the $(m-1)$-step factor of 
$X_i$, $q_i\colon X_i\to W_i$ the 
factor map and $r_i=q_i\circ\pi_i$.

 We have showed above that, for every $i$, the measure ${\mu_i}\cube{m}_e$ is relatively 
independent with respect to ${\sigma_i}\cube{m}_e$.

 By using c) 
twice, we get that the second integral in~\eqref{eq:XiX}
is equal to
$$
 \int\prod_{\emptyset\neq\epsilon\subset[d]}
\E(f_\epsilon\mid W_i)\circ r_i (x_\epsilon)
\,d\mu\cube{m}_e(\bx)\ .
$$
As the systems $X_i$ form an increasing sequence, the systems 
$W_i$ also form an increasing sequence. Let $(W,\sigma,T)$ be the 
inverse limit if this sequence. This system is a factor of $X$, and 
writing $r\colon X\to W$ for the factor map, we have that
$\E(f_\epsilon\mid W_i)\circ r_i\to\E(f_\epsilon\mid W)\circ r$ in 
$L^1(\mu)$ for every 
$\epsilon$. We get
\begin{equation}
\label{eq:WX}
\int\prod_{\emptyset\neq\epsilon\subset[d]}f_\epsilon(x_\epsilon)
\,d\mu\cube{m}_e(\bx)=
 \int\prod_{\emptyset\neq\epsilon\subset[d]}\E(f_\epsilon\mid W)\circ 
r(x_\epsilon)
\,d\mu\cube{m}_e(\bx)\ .
\end{equation}
This means that the measure $\mu\cube{m}_e$ is relatively 
independent with respect to $\sigma\cube{m}_e$.

As $W$ is an inverse limit of $(m-1)$-step nilsystems and is a factor of 
$X$, it is a factor of the $(m-1)$-step factor $Y$ if $X$. If for some 
$\epsilon$ we have $\E(f_\epsilon\mid Y)=0$, then  we have 
$\E(f_\epsilon\mid  W)=0$ and the second integral in~\eqref{eq:WX} is 
equal to zero. The result follows.
\end{proof}

Passing to inverse limits adds technical issues to each proof.  These 
issues are not difficult and the passage to inverse limits uses only 
routine techniques, as in the preceding proof.  However, it does greatly increase the length of 
the arguments, and so in general we omit this portion of the argument.

\subsection{The measures $\mu\cube m_{e, x}$}\strut

In this section, again $(X,\mu,T)$ is an ergodic inverse 
limit of $k$-step nilsystems, with base point $e\in X$.

For $x\in X$ we write  
$$
 X\cube m_{e, x}=\{\bx\in X\cube m\colon x_\emptyset=e\text{ and }
x_{\{m\}}=x\}\ .
$$
The set $X\cube m_{\bbullet}$ is the image of the set $X\cube{m,1}_e$ introduced 
below
by a permutation of coordinates.

\begin{proposition}
\label{prop:muax}
For each $x\in X$, there exists a measure $\mu\cube m_{e, x}$, 
concentrated on $X\cube m_{e, x}$, such that
\begin{enumerate}
\item
The image of $\mu\cube m_{e, x}$ under each projection 
$\bx\mapsto x_\epsilon\colon X\type m\to X$, $\epsilon\neq\emptyset$, 
$\epsilon\neq\{m\}$,
 is equal to $\mu$.
\item
If $f_\epsilon$, $\epsilon\subset[m]$, $\epsilon\not\subset[1]$, are 
$2^m-2$ bounded measurable functions on $X$, then the function $F$ on 
$X$ given by 
$$
 F(x)=\int\prod_{\substack{\epsilon\subset[d] \\ 
\epsilon\neq\emptyset,\ \epsilon\neq\{m\} }}
f_\epsilon(x_\epsilon)\,d\mu_{e, x}\cube m(\bx)
$$
is continuous. 
\item Moreover, for every bounded measurable function $f$ on 
$X$,
$$
 \int f(x)F(x)\,d\mu(x)=\int f(x_{ \{m\} })\,
\prod_{\substack{\epsilon\subset[d] \\ 
\epsilon\neq\emptyset,[m]}}
f_\epsilon(x_\epsilon)\,d\mu\cube m_e(\bx)\ .
$$
\end{enumerate}
\end{proposition}

\begin{proof}
It suffices to prove this Proposition in the case that $(X,\mu,T)$ is 
$k$-step nilsystem, as the general case follows by standard methods.

We write $X=G/\Gamma$ as usual. We can assume that $e$ is the image 
in $X$ of the unit element $1$ of $G$. We define
$$
G\cube m_{\bbullet
}=
\bigl\{\bg\in G\cube m\colon g_\emptyset=g_{ \{m\} }=1\bigr\}\  .
$$
This group is closed and normal in $G$. 
It is the set of elements of $G\cube m$ that can
 be written 
as in~\eqref{eq:deffacegroup} with $g_\emptyset=1$
and $g_i=1$ for the value of $i$ such that $\alpha_i=\{m\}$.
Recall that $e\type m = (e, e, \ldots, e)$.  

It is easy to check that $G\cube m_{\bbullet
}\cdot e\type 
m=X\cube m_{\bbullet
}$. It follows that 
$$
 \Gamma\cube m_{\bbullet
}:=\Gamma\type m\cap G\cube m_{\bbullet
}
$$
is cocompact in $G\cube m_{\bbullet
}$ and that $X\cube
m_{\bbullet
}$ can be identified with the nilmanifold 
$G\cube m_{\bbullet
}/\Gamma\cube m_{\bbullet
}$.  
We write 
$\mu\cube m_{\bbullet
}$ for the Haar measure of this nilmanifold.

Let $F = \{\epsilon\subset[m]\colon m\in\epsilon\}$.  We recall that 
for $g\in G$, $g^F\in G\cube m$ is defined by 
$$
(g^F)_\epsilon = 
\begin{cases}
g & \text{ if } m\in\epsilon \\
1 & \text{ otherwise.}
\end{cases}
$$

By definition of the sets $X\cube m_{e, x}$,
 the image of $X\cube m_{\bbullet
}$ under translation by 
$g\type m_m$ is equal to $X\cube m_{e, g\cdot e}$. 
Since $G\cube m_{\bbullet
}$ is normal in $G\cube m$, 
the image of the measure $\mu_{\bbullet
}\cube m$ under 
$g^F$ is invariant under $G\cube m_{\bbullet
}$. 
Moreover, if $g,h\in G$ satisfy $g\cdot e=h\cdot e$, 
then we have that $g=h\gamma$ for some $\gamma\in\Gamma$. 
Since $\gamma^F\cdot e\type m=e\type m$ and by normality of 
$G\cube m_{\bbullet
}$ again, the measure $\mu\cube m_{\bbullet
}$ 
is invariant under $\gamma^F$ and thus the images of 
$\mu\cube m_{\bbullet
}$ under $g^F$ and $h^F$ 
are the same.

Therefore, for every $x\in X$ we can define a measure 
$\mu\cube m_{e, x}$ on $X\cube m_{e, x}$ by 
\begin{equation}
\label{eq:muaxg}
\mu\cube m_{e, x}=g^F\cdot \mu\cube m_{\bbullet
}\text{ for every 
}g\in G\text{ such that }g\cdot e =x\ .
\end{equation}
In particular, for every $h\in G$ and every $x\in X$,
\begin{equation}
\label{eq:muaxh}
\mu_{e, h\cdot x}\cube m=h^F\cdot \mu\cube m_{e, x}\ .
\end{equation}
If $T$ is the translation by $\tau\in G$, then $T\type m_m$ is the 
translation by $\tau^F$ and 
and so for every integer $n$,
\begin{equation}
\label{eq:maxTn}
\mu_{e, T^n x}\cube m={T\type m_m}^n\cdot \mu\cube m_{e, x}\ .
\end{equation}
For $1\leq i < m$, $\tau\type m_i\in G\cube m_{\bbullet
}$ 
and thus, for every $x\in X$,  $\mu\cube m_{e, x}$ is invariant 
under $T\type m_i$.
As above, it follows that this measure satisfies the first property of 
the proposition.

To prove the other properties, the first statement of the 
proposition implies that we can reduce to the case that the 
functions $f_\epsilon$ are continuous. By~\eqref{eq:muaxg}, the map 
$x\mapsto \mu\cube m_{e, x}$ is weakly continuous  and the function $F$ 
is continuous. We are left with showing that 
$$
 \mu\cube m_e =\int\mu\cube m_{e, x}\, d\mu(x)\ .
$$
For $1\leq i < m$, since for every $x$ the measure 
$\mu\cube m_{e, x}$ 
is invariant under $T_i\type m$, the measure defined by this integral 
is invariant under this transformation. 
By~\eqref{eq:muaxh}, 
$\mu\cube m_{e, Tx}=T\type m_m\cdot\mu\cube m_{e, x}$ 
for every $x$ and it follows that the measure defined by the 
above integral is invariant under $T\type m_m$.  
Since it is concentrated on $X\cube m_e$, it is equal to the 
Haar measure $\mu\cube m_e$ of this nilmanifold (recall that 
$(X\cube m_e,T\type m_1,\dots,T\type m_m)$ is uniquely ergodic).
\end{proof}

\subsection{A positivity result}
In this section, again $(X,\mu,T)$ is an ergodic inverse limit of $k$-step nilsystems, 
with base point $e\in X$. 

In the next proposition, the notation 
$\epsilon=\epsilon_1\dots\epsilon_m\in\{0,1\}^m$ is more convenient 
that $\epsilon\subset[m]$. We recall that $00\dots 0\in\{0,1\}^m$ 
corresponds to $\emptyset\subset[m]$ and that $00\dots 01\in\{0,1\}^m$ corresponds to 
$\{m\}\subset[m]$. For $\epsilon \in\{0,1\}^{m+1}$, 
$\epsilon_1\dots\epsilon_m$ corresponds to $\epsilon\cap[m]$.
\begin{proposition}
\label{prop:positive2}
Let $f_\epsilon$, $\emptyset\neq\epsilon\in\{0,1\}^m$, be
 $2^m-1$  bounded measurable real functions on $X$.
Then 
\begin{multline*}
 \int\prod_{\substack{\epsilon\in\{0,1\}^{m+1}\\ 
\epsilon\neq 00\dots 0\\ \epsilon\neq 00\dots 01 }}
f_{\epsilon_1\dots\epsilon_m }(x_\epsilon)\,d\mu\cube{m+1}_\bbullet
(\bx)\\
\geq
\Bigl(\int\prod_{\substack{\epsilon\in\{0,1\}^m\\ 
\epsilon\neq 00\dots 0 }}
f_\epsilon(x_\epsilon)\,d\mu\cube m_e(\bx)\Bigr)^2\ .
\end{multline*}
\end{proposition}
\begin{proof}

 We first reduce the general case to that of 
an ergodic $k$-step nilsystem.  If
$(X,\mu,T)$ is an inverse limit of an increasing sequence of $k$-step 
ergodic nilsystems, then the spaces $X\cube m_e$ and 
$X\cube{m+1}_\bbullet$, as well as the measures $\mu\cube m_e$ and 
$\mu\cube{m+1}_\bbullet$, are the inverse limits of the corresponding 
objects associated to each of the nilsystems in the sequence of 
nilsystems converging to $X$. Thus it suffices to prove the 
proposition when $(X,T,\mu)$ is an ergodic $k$-step nilsystem. We 
write $X=G/\Gamma$ as usual.

 The groups $G\cube m_e$, $\Gamma\cube m_e$,
$G\cube {m+1}_{\bbullet
}$ and $\Gamma\cube {m+1}_{\bbullet
}$
have been defined and studied above. We recall that $X\cube 
m_{e }=G\cube m_e/\Gamma\cube m_e$ and that 
$\mu\cube m_e$ is the Haar measure of this nilmanifold. Also,
$X\cube{m+1}_{\bbullet
}=G\cube {m+1}_{\bbullet
}/ 
\Gamma\cube {m+1}_{\bbullet
}$ and 
$\mu\cube{m+1}_{\bbullet
}$ is the Haar measure of this 
nilmanifold.

It is convenient to identify $X\type{m+1}$ with
 $X\type m\times X\type m$, writing a point $\bx\in X\type{m+1}$ as 
$\bx=(\bx',\bx'')$, where 
$$
\bx'=(x_{\epsilon_1\dots\epsilon_m0}\colon 
\epsilon\in\{0,1\}^m)
\text{ and }
\bx''=(x_{\epsilon_1\dots\epsilon_m1}\colon
\epsilon\in\{0,1\}^m)\ .
$$ 

The \emph{diagonal map} $\Delta\cube m_X\colon X\type m\to X\type 
{m+1}$ is defined by $\Delta\cube m_X(\bx)=(\bx,\bx)$, that is,
$$
\text{for }\bx\in X\type m\text{ and }\epsilon\in\{0,1\}^{m+1},\quad
 \bigl(\Delta\cube m_X(\bx)\bigr)_\epsilon
=x_{\epsilon_1,\dots,\epsilon_m}\ .
$$
We remark that $\Delta\cube m_X(X\cube m_e)\subset 
X\cube{m+1}_\bbullet$.

We use similar notation for elements of $G\type{m+1}$ and define the 
diagonal map $\Delta\cube m_G\colon   G\type {m}\to 
G\type {m+1}$.  
We have 
$$
\Delta\cube m(G_e\cube m)\subset G\cube{m+1}_{\bbullet
}
$$
and, for every $\bg=(\bg',\bg'')\in G\cube{m+1}_{\bbullet
}$ we have that $\bg'$ and $\bg''$
belong to $G\cube m_e$; 
in other words,
$G\cube{m+1}_{\bbullet
}\subset G\cube m_e\times G\cube 
m_e$.
We define 
$$
 G\cube{m}_{*}=\{ \bg\in G\cube{m}_e
\colon 
(1\type m,\bg)\in G\cube{m+1}_\bbullet\}
$$
and we have that $G\cube{m}_{*}$ is a closed normal subgroup of 
$G\cube{m}_e$ and that
$$
 G\cube{m+1}_\bbullet
=\bigl\{(\bg,\bh\bg)\colon \bg\in G\cube 
m_e,\ \bh\in G\cube m_*\bigr\}\ .
$$
It follows that
$$
 X\cube{m+1}_\bbullet
=\bigl\{(\bx,\bh\cdot\bx)\colon
\bx\in X\cube m_e,\ \bh\in G\cube{m+1}_{*}\bigr\}\ .
$$

For every $\bx'\in X\cube{m}$, set 
$$
V_{\bx'} = \{\bx''\in X\cube{m}\colon (\bx', \bx'')\in X\cube{m+1}_\bbullet 
\}\ .
$$

For $\bx\in X\cube{m}$ and $\bg\in G\cube{m}_e$ we have
\begin{equation}
\label{eq:nugx}
\text{the image of }\nu_{\bx}\text{ under translation by }\bg
\text{ is equal to }\nu_{\bg\cdot\bx}\ .
\end{equation}
Indeed, this image  is supported on $\nu_{\bg\cdot\bx}$ and 
is invariant under $G\cube{m}_*$, since $G\cube{m}_*$ is normal 
in $G\cube{m}_e$.  

We claim that 
\begin{equation}
\label{eq:int-delta-nu}
\mu\cube{m+1}_\bbullet = 
\int\delta_{\bx}\times\nu_{\bx}\,d\mu\cube{m}_e(\bx) \ .
\end{equation}
The measure on $X\cube{m+1}_\bbullet$ defined by this integral 
is invariant under translation by elements of the form $(1\type 
m, \bh)$ with $\bh\in G\cube{m}_*$ (note that each $\delta_\bx\times\nu_{\bx}$ 
is invariant under such translations).   By~\eqref{eq:nugx}, the measure 
defined by this integral is also invariant under translation by $(\bg, 
\bg)$ for $\bg\in G\cube{m}_e$. Therefore this measure is 
invariant under $G\cube{m+1}_\bbullet$.  Since it is supported on 
$X\cube{m+1}_\bbullet$, it is equal to the Haar measure 
$\mu\cube{m+1}_\bbullet$  of this nilmanifold.  The claim 
is proven. 

By~\eqref{eq:nugx} again, $\nu_{\bh\cdot\bx} = \nu_{\bx}$  for $\bh\in G\cube{m}_*$.
Let $\CF$ denote the $\sigma$-algebra of $G\cube{m}_*$-invariant 
functions. 
For every bounded Borel function $F$ on $X\cube{m}_e$, 
\begin{equation}
\label{eq:int-F}
\int F\,d\nu_\bx = \E(F\mid\CF)(\bx)\quad 
\mu\cube{m}_e\text{-a.e.}
\end{equation}
To see this, we note that the 
function defined by this integral is invariant under translation 
by $G\cube{m}_*$ and thus is $\CF$-measurable.  Conversely, if $F$ 
is $\CF$-measurable, then for $\mu\cube{m}_e$ almost every 
$\bx$, it coincides $\nu_{\bx}$-almost everywhere with a constant and so 
the integral is equal almost everywhere to $F(\bx)$.

Thus for a bounded Borel function $F$ on $X\cube{m}_e$,  
using~\eqref{eq:int-delta-nu} and~\eqref{eq:int-F}, we have that 
\begin{align*}
\int F(\bx')F(\bx'')\, d\mu\cube{m+1}_\bbullet(\bx) & = 
\int\Bigl(F(\bx')\int F(\bx'')\,d\nu_{\bx'}(\bx'')
\Bigr)\, d\mu\cube{m}_e(\bx') \\
& = \int F\cdot \E(F\mid\CF)\,d\mu\cube{m}_e \\
& =\int\E(F\mid\CF)^2\,d\mu\cube{m}_e
\geq\Bigl( \int F\,d\mu\cube{m}_e \Bigr)^2\ .
\end{align*}

\end{proof}

\subsection{The measures $\mu\cube{m,r}_e$}

In this section again, $(X,\mu,T)$ is an ergodic inverse limit of 
$k$-step nilsystems, with base point $e\in X$. 
Let $m$ and $r$ be integers with $0\leq r<m$.

Let $\Delta_{m,r}\colon X\type {m-r}\to X\type{m}$ be the map given by
\begin{multline*}
 \text{for }\bx\in X\type {m-r}\text{ and }
\epsilon=\epsilon_1,\dots,\epsilon_{m}\in\{0,1\}^{m},\\
(\Delta_{m,r}\bx)_\epsilon=x_{\epsilon_{r+1}\dots\epsilon_{m}}\ .
\end{multline*}
We define:
\begin{gather}
\label{eq:Xamr}
 X\cube{m,r}_e=\Delta_{m,r}\bigl(X\cube {m-r}_e\bigr) \text{ and }\\
\label{eq:muamr}
\mu\cube{m,r}_e\text{ is the image of }\mu\cube {m-r}_e
\text{ under 
}\Delta_{m,r} \ .
\end{gather}

Recall that $X\cube {m-r}_e$ is the closed orbit of $e\type{m-r}$ 
under the transformations $T_i\type {m-r}$ for $1\leq i\leq m-r$ and 
that $\mu\cube {m-r}_e$ is the unique probability measure of this set 
invariant under these transformations.
We have  $\Delta_{m,r}e\type {m-r}=e\type m$, and, for $1\leq i\leq 
m-r$, 
$\Delta_{m,r}\circ T\type  {m-r}_i=T\type{m}_{r+i}\circ\Delta_{m,r}$. 
Therefore:

\medskip\noindent\emph{
 $X\cube{m,r}_e$ is the closed orbit of the point 
$e\type{m}\in X\cube{m}$ under the transformations 
$T\type{m}_i$ for $r+1\leq i\leq m$ and 
 $\mu_e\cube{m,r}$ is the unique probability measure on this set 
invariant under these transformations.
}

For example, $X\cube{m,0}_e=X\cube m_e\subset X\cube m$ and 
$\mu\cube{m,0}_e=\mu\cube m_e$.

$X\cube{r+1,r}_e=\{e\type r\}\times \Delta\type r\subset X\cube{r+1}$, 
where $\Delta\type r$ 
denotes the diagonal of $X\type r$. $\mu\cube{r+1,r}_{e }$ is the 
product of the Dirac mass at $e\type r$ by the diagonal measure of $X\type r$.

Since the image of $\mu\cube {m-r}_e$ under the projections 
$\bx\mapsto x_\epsilon$ with $\epsilon\neq\emptyset$, are equal to 
$\mu$, we have that:

\medskip\noindent\emph{
The images of $\mu\cube{m,r}_e$ under the projections $\bx\mapsto
x_\epsilon$ for $\epsilon\subset[m]$, $\epsilon\not\subset[r]$, are 
equal to $\mu$.
}

 Therefore,
if $h_\epsilon$, $\epsilon\subset[m]$, 
$\epsilon\not\subset [r]$, are $2^{m}-2^r$  measurable 
functions on $X$ with $|h_\epsilon|\leq 1$, we have that
\begin{equation}
\label{eq:projmu}
 \Bigl|\int\prod_{\substack{\epsilon\subset[m]\\ 
\epsilon\not\subset[r]}}h_\epsilon(x_\epsilon)\,
d\mu\cube{m,r}_{e }(\bx)\Bigr|
\leq\min_{\substack{\epsilon\subset[m]\\ \epsilon\not\subset[r]}}
\norm{h_\epsilon}_{L^1(\mu)}\ .
\end{equation}

\section{A convergence result}
\label{sec:induction}

In this section, we prove the key convergence result 
(Proposition~\ref{prop:ouf}).

\subsection{Context}
\label{subsec:context}
We recall our context, as introduced in 
Sections~\ref{subsec:Furstenberg} and~\ref{subsec:uniformity}.

The system $(X,T)$ is associated to the subalgebra $\CA$ of 
$\ell^\infty(\Z)$, $\mu$ is an ergodic invariant probability measure 
on $X$, associated to the averages on the sequence $\bI=(I_j\colon 
j\geq 1)$ of intervals.

For every $k\geq 1$, let $(Z_k,\mu_k,T)$ be the factor of order $k$ 
of $(X,\mu,T)$. We recall that this system is an inverse limit of 
$(k-1)$-step nilsystems, both in the topological and the ergodic 
theoretical senses. The system $(Z_k,T)$  is distal, minimal 
and uniquely ergodic, and $Z_k$ is given with a base point $e_k$.
In a futile attempt to keep the notation only mildly disagreeable, 
when the base point $e_k$ is used as
a subindex, we omit the subscript $k$.

We write $\pi_k\colon X\to Z_k$ for the factor 
map. We recall that this map is measurable, and has no reason for 
being continuous.
For $\ell\leq k$, $Z_\ell$ is a factor of $Z_k$, with a factor map
$\pi_{\ell, k}\colon Z_k\to Z_\ell$ which is continuous and 
 $\pi_{\ell,k}(e_k)=e_\ell$.

We use various different methods of taking limits of averages of 
sequences indexed by $\Z^r$.  For example, in 
Proposition~\ref{prop:zerodensity}, we average over any F\o lner 
sequence in $\Z^r$.  In the sequel, we use iterated limits: if 
$\bigl(a_n\colon n=(n_1, \ldots, n_r)\in\Z^r\bigr)$ is a bounded sequence, we define the iterated 
$\limsup$ of $a$ as
\begin{multline*}
\iter \limsup\vert\Av_{\bI,n_1, \ldots, n_r} a_{n_1, \ldots, n_r}\vert  \\ = 
\limsup_{j_1\to\infty}\ldots\limsup_{j_{r}\to\infty}
\frac{1}{|I_{j_1}|\ldots|I_{j_r}|}\Bigl\vert 
\sum_{\substack{n_1\in I_{j_1}\\ \ldots \\ {n_r\in I_{j_r}}}}a_{n_1, 
\ldots, n_r}
\Bigr\vert
\ .
\end{multline*}

We define the $\iter\lim\Av a_n$ analogously, assuming that all of 
the limits exist.  

\subsection{An upper bound}
The next proposition is proved in Section $13$ of~\cite{HK}:
\begin{proposition}
\label{prop:zerodensity}
Let $(X,\mu,T)$ be an ergodic system and $(Z_d,T,\nu)$ be its factor of 
order $d$. Let $f_\epsilon$, 
$\epsilon\subset [d]$, be $2^d$ bounded measurable functions on $X$. 
For $n=(n_1,\dots,n_d)\in\Z^d$, let
$$
 a_n=\int\prod_{\epsilon\subset[d]}f_\epsilon(T^{n\cdot\epsilon}x)\,d\mu(x)
\text{ and }
b_n=\int\prod_{\epsilon\subset[d]}
\E(f_\epsilon\mid Z_d)(T^{n\cdot\epsilon}z)\,d\mu_d(z)\ .
$$
Then $a_n-b_n$ converges to zero in density, meaning that the averages 
of $[a_n-b_n|$ on any F\o lner sequence in $\Z^d$ converge to zero.
\end{proposition}
\begin{lemma}
\label{lem:boundkr}
Let $k\geq 1$, $0\leq r\leq d$ and $h_\epsilon$, 
$\epsilon\subset[d+1]$, $\epsilon\not\subset[r]$, be $2^{d+1}-2^r$ 
continuous functions on $Z_k$. Then for every $\delta>0$, there exists 
$C=C(\delta)>0$ with the following property:

Let $\psi_\epsilon$, $\epsilon\subset[r]$, be 
$2^{r}$ sequences belonging to $\CB$ with  absolute value $\leq 1$. 
Then the iterated $\limsup$ in $n_1,\dots,n_r$ of the absolute value of the averages 
on $\bI$ of 
$$
 A(n):=\prod_{\epsilon\subset[r]}\psi_\epsilon(n\cdot\epsilon)
\int \prod_{\substack{\epsilon\subset[d+1]\\ \epsilon \not\subset[r]}}
h_\epsilon(T^{n\cdot\epsilon}x_\epsilon)\,d\mu\cube{d+1,r}_{k\ e}(\bx)
$$
is bounded by 
$$
 \delta+C\prod_{r\in\epsilon\subset[r]}\norm{\psi_\epsilon}_{\bI,k+r}\ .
$$
\end{lemma}

\begin{proof}
We write $n=(m_1,\dots,m_{r-1},p)$ and $m=(m_1,\dots,m_{r-1})$.
The expression  to be averaged can be rewritten as
\begin{multline*}
A'(m,p)=\\
 \prod_{\epsilon\subset[r-1]}
\psi_\epsilon(m\cdot\epsilon)\cdot
\prod_{r\in\epsilon\subset[r]}
\psi_\epsilon(m\cdot\epsilon+p)\cdot
\int \prod_{\substack{\epsilon\subset[d+1]\\ \epsilon \not\subset[r]}}
f_\epsilon(T^{m\cdot\epsilon+p\epsilon_r}x_\epsilon)
\,d\mu_{k\ e}\cube{d+1,r}(\bx)\ .
\end{multline*}
For $m\in\Z^{r-1}$, we write 
$$
 \Phi_m(p)= 
\prod_{r\in\epsilon\subset[r]}\psi_\epsilon(m\cdot\epsilon+p)=
\prod_{r\in\epsilon\subset[r]}\sigma^{m\cdot\epsilon}\psi_\epsilon(p)
$$
where $\sigma$ is the shift on $\ell^\infty(\Z)$.
For $\bx\in Z_k\cube{d+1,r}$, we also write 
$$
 H(\bx)=\prod_{\substack{\epsilon\subset[d+1]\\ \epsilon 
\not\subset[r]}}h_\epsilon(x_\epsilon)
$$
and for every $\delta>0$, we let $C=C(\delta)$ be associated to this 
continuous function on $X_k\cube{d+1,r}$ as in 
Proposition~\ref{prop:boundphif}.
We have
\begin{multline*}
 \prod_{r\in\epsilon\subset[r]}
\psi_\epsilon(m\cdot\epsilon+p)\cdot\prod_{\substack{\epsilon\subset[d+1]\\
 \epsilon \not\subset[r]}}
h_\epsilon(T^{m\cdot\epsilon+p\epsilon_r}x_\epsilon)\\
=\Phi_m(p)H\bigl(T_r^{[d+1]p}(T_1^{[d+1]m_1}\dots 
T_{r-1}^{[d+1]m_{r-1}}\bx)\bigr)
\end{multline*}
and thus
\begin{multline*}
 \Bigl|\limsup_j \Av_{p\in I_j}
\prod_{r\in\epsilon\subset[r]}\psi_\epsilon(m\cdot\epsilon+p)
\cdot
\prod_{\substack{\epsilon\subset[d+1]\\ \epsilon \not\subset[r]}}
h_\epsilon(T^{m\cdot\epsilon+p\epsilon_r}x_\epsilon) \Bigr|\\
\leq 
\delta+C\norm{\Phi_m}_{\bI,k+1}
\end{multline*}
for every $m$ and every $\bx\in Z_k\cube{d+1,r}$. Taking the integral,
$$
  \bigl|\limsup_j \Av_{p\in I_j}A'(m,p)\bigr|\leq \delta+C\norm{\Phi_m}_{\bI,k+1}
$$
for every $m$.  Therefore
\begin{multline*}
 \iter\limsup\bigl|\Av_{\bI,n_1,\dots,n_r}A(n)\bigr|\\
\leq  
\iter\limsup\Av_{\bI,n_1,\dots,n_{r-1}}
\bigl|\lim_j \Av_{p\in I_j}A'(m,p)\bigr|\\
\leq\delta+C\iter\limsup\Av_{\bI,m_1,\dots,m_{r-1}}\norm{\Phi_m}_{\bI,k+1}\\
\leq\delta+C\iter\limsup
\Bigl(\Av_{\bI,m_1,\dots,m_{r-1}}\norm{\Phi_m}_{\bI,k+1}^{2^{r-1}}\Bigr)^{1/2^{r-1}}\ .
\end{multline*}
By a result in~\cite{HK3}, the last $\limsup$ is 
actually a limit  and is 
bounded by
$$
\prod_{r\in\epsilon\subset[r]}\norm{\psi_\epsilon}_{\bI,k+r}\ .\qed
$$
\renewcommand{\qed}{}
\end{proof}

\subsection{Iteration}
\begin{proposition}
\label{prop:recur}
Let $k\geq 1$, $0\leq r\leq d$ and $h_\epsilon$, 
$\epsilon\subset[d+1]$, $\epsilon\not\subset[r]$, be $2^{d+1}-2^r$ 
bounded measurable functions on $Z_k$.  Let $\phi_\epsilon$, 
$\epsilon\subset[r]$, be $2^r$ continuous functions on $X$. For 
$n\in\Z^r$, define
$$
 A(n)=\prod_{\epsilon\subset[r]}\psi_\epsilon(n\cdot\epsilon)
\int \prod_{\substack{\epsilon\subset[d+1]\\ \epsilon \not\subset[r]}}
h_\epsilon(T^{n\cdot\epsilon}x_\epsilon)\,d\mu_{k\ e}\cube{d+1,r}(\bx)
$$
and
\begin{multline*}
 B(n)=
\prod_{\epsilon\subset[r-1]}\phi_\epsilon(T^{n\cdot\epsilon}x_0)\cdot\\
\int 
\prod_{r\in\epsilon\subset[r]}\E(\phi_\epsilon\mid 
Z_{k+r-1})(x_\epsilon)\cdot
\prod_{\substack{\epsilon\subset[d+1]\\ \epsilon 
\not\subset[r]}}h\circ p_{k+r-1,k}(x_\epsilon)\,
d\mu_{k+r-1\ e}\cube{d+1,r-1}(\bx)\ .
\end{multline*}
Then the iterated limit of the averages of $A(n)-B(n)$ is zero.
\end{proposition}

\begin{proof}
We remark that $B(n)$ 
depends only on $n_1,\dots,n_{r-1}$.

By~\eqref{eq:projmu}, it suffices to prove the result in the 
case that the functions $f_\epsilon$ are continuous. We can also 
assume that $|\phi_\epsilon|\leq 1$ for every $\epsilon\subset[r]$.

Let $\delta>0$ be given and let $C$ be as in Lemma~\ref{lem:boundkr}.
For each $\epsilon$ with $r\in\epsilon\subset[d+1]$, let $\widetilde\phi_\epsilon$ 
be a continuous function on $Z_{k+r-1}$ with $|\widetilde\phi_\epsilon|\leq 1$, 
such that 
$\norm{\E(\phi_\epsilon\mid Z_{k+r-1})-\widetilde\phi_\epsilon}$ is sufficiently 
small. We have that
$$
 \norm{(\widetilde\phi_\epsilon(T^ne_{k+r-1})\colon n\in\Z)-
(\phi_\epsilon(T^nx_0)\colon n\in\Z)}_{\bI,k+r}\leq \delta/2^{r-1}C
$$
for every $\epsilon$.  This follows from Proposition~\ref{prop:approx}.

By Lemma~\ref{lem:boundkr} the iterated $\limsup$ of the absolute value of the averages 
on $\bI$ of 
\begin{multline*}
A(n)-  \\ \prod_{\epsilon\subset[r-1]}\phi_\epsilon(T^{n\cdot\epsilon}x_0)\cdot
\prod_{r\in\epsilon\subset[r]}\widetilde\phi_\epsilon(T^{n\cdot\epsilon}e_{k+r-1})\cdot
\int \prod_{\substack{\epsilon\subset[d+1]\\ \epsilon \not\subset[r]}}
h_\epsilon(T^{n\cdot\epsilon}x_\epsilon)\,d\mu\cube{d+1,r}_{k\ e}(\bx)
\end{multline*}
is bounded by $2\delta$.
We rewrite the second term in this difference as
\begin{multline*}
\prod_{\epsilon\subset[r-1]}\phi_\epsilon(T^{n\cdot\epsilon}x_0)\cdot 
\prod_{r\in\epsilon\subset[r]}\widetilde\phi_\epsilon(T^{n\cdot\epsilon}e_{k+r-1})\cdot \\
\int \prod_{\substack{\epsilon\subset[d+1]\\ \epsilon \not\subset[r]}}
h_\epsilon\circ 
p_{k+r-1,r}(T^{n\cdot\epsilon}x_\epsilon)\,d\mu_{k+r-1\ e}\cube{d+1,r}(\bx)
\end{multline*}
and remark that the first product in this last expression depends only on 
$n_1,\dots,n_{r-1}$. 

By definition of the measures, the averages in 
$n_r$ on $\bI$ of the above expression converges to
$$
\prod_{\epsilon\subset[r-1]}\phi_\epsilon(T^{n\cdot\epsilon}x_0)\cdot
\int \prod_{r\in\epsilon\subset[r]}\widetilde\phi_\epsilon(x_\epsilon)\cdot
\prod_{\substack{\epsilon\subset[d+1]\\ \epsilon 
\not\subset[r]}}h_\epsilon\circ p_{k+r-1,k}(x_\epsilon)\,
d\mu_{k+r-1}\cube{d+1,r-1}(\bx)\ .
$$
By~\eqref{eq:projmu} again, for every $n_1,\dots,n_{r-1}$ the difference between this expression 
and $B(n)$ is bounded by $\delta$.

The announced result follows.
\end{proof}

\begin{proposition}
\label{prop:ouf}
Let $k\geq 1$  and let $f_\epsilon$, 
$\epsilon\subset[d+1]$,  $\epsilon\neq\emptyset$, be $2^{d+1}-1$ 
continuous functions on $X$.
Then the iterated averages for $n=(n_1,\dots,n_d, n_{d+1})\in\Z^{d+1}$ 
on $\bI$ of
\begin{equation}
\label{eq:prop1bis}
 \prod_{\emptyset\neq\epsilon\subset[d+1]}
f_\epsilon(T^{n\cdot\epsilon}x_0)
\end{equation}
converge to
\begin{equation}
\label{eq:prop2}
 \int  \prod_{\emptyset\neq\epsilon\subset[d+1]}
\E(f_\epsilon\mid Z_d)(x_{\epsilon})
\,d\mu_{d\ e}\cube{d+1}(\bx)\ .
\end{equation}
\end{proposition}

\begin{proof}
For notational convenience we define $f_\emptyset$ to be the constant 
function $1$.

By~\eqref{eq:intphi}, the averages in $n_{d+1}$ 
of~\eqref{eq:prop1bis} converge to
\begin{equation}
\label{eq:prop1}
 \prod_{\emptyset\neq\epsilon\subset[d]}f_\epsilon(T^{n\cdot\epsilon}x_0)\cdot
\int \prod_{\substack{\epsilon\subset[d+1] \\\epsilon\not\subset[d]}}
f_\epsilon(T^{n\cdot\epsilon}x)
\,d\mu(x)
\end{equation}
and it remains to show that the iterated averages in 
$(n_1,\dots,n_d)$ of this expression converge to~\eqref{eq:prop2}.

By Proposition~\ref{prop:zerodensity}, the difference between the 
quantity~\eqref{eq:prop1} and 
$$
 A(n):=\prod_{\epsilon\subset[d]}
f_\epsilon(T^{n\cdot\epsilon}x_0)\cdot
\int \prod_{\substack{\epsilon\subset[d+1] \\ \epsilon\not\subset[d]}}
\E(f_\epsilon\mid Z_d)(T^{n\cdot\epsilon}x)\,d\mu_d(x)
$$
converges to zero in density and we are reduced to study the 
iterated convergence of the averages of $A(n)$.

We apply Proposition~\ref{prop:recur} 
with $k=d$ and $r=d$ and left with studying the iterated limit of the 
averages in $n_1,\dots,n_{d-1}$ of
\begin{multline*}
 \prod_{\epsilon\subset[d-1]}f_\epsilon(T^{n\cdot\epsilon}x_0)\cdot\\
\int 
\prod_{d\in\epsilon\subset[d]}
\E(f_\epsilon\mid Z_{2d-1})(x_\epsilon)\cdot
\prod_{\substack{\epsilon\subset[d+1]\\ \epsilon 
\not\subset[d]}}
\E(f_\epsilon\mid Z_d)\circ p_{2d-1,d}(x_\epsilon)\,
d\mu_{2d-1\ e}\cube{d+1,d-1}(\bx)\ .
\end{multline*}

After $d-r$ steps, we are left with the iterated limit of the 
averages in $n_1,\dots,n_r$ of an expression of the form
$$
 \prod_{\epsilon\subset[r]}f_\epsilon(T^{n\cdot\epsilon}x_0)\cdot\\
\int 
\prod_{\substack{\epsilon\subset[d+1]\\ \epsilon\not\subset [r]}}
E(f_\epsilon\mid Z_{\ell(\epsilon)})\circ 
p_{k,\ell(\epsilon)}(x_\epsilon)\,d\mu_{k\ e}\cube{d+1,r}(\bx)\ ,
$$
where $k=k(r)\geq d$ is an integer and where for every $\epsilon$, 
$d\leq\ell(\epsilon)\leq k$.

Finally, after $d$ steps, we  have that the iterated limit of the 
expression~\eqref{eq:prop1}  exists and is equal to
$$
 \int  \prod_{\emptyset\neq\epsilon\subset[d+1]}
\E(f_\epsilon\mid Z_{\ell(\epsilon)})\circ p_{k,\ell(\epsilon)}(x_{\epsilon})
\,d\mu_{k\ e}\cube{d+1}(\bx)\ ,
$$
where $k$ is an integer and $d\leq \ell(\epsilon)\leq k$ for every 
$\epsilon$.

By Proposition~\ref{prop:mesuresmua}, the measure $\mu\cube 
{d+1}_{k\ e}$ is relatively independent with respect to its projection 
$\mu\cube{d+1}_{d\ e}$ on $Z_d\cube {d+1}$. For every $\epsilon$,
$$
 \E\bigl( \E(f_\epsilon\mid Z_{\ell(\epsilon)})\circ p_{k,\ell(\epsilon)}
\mid Z_d\bigr)=\E(f_\epsilon\mid Z_d)
$$
and we have that the above limit is equal to~\eqref{eq:prop2}.
\end{proof}

\section{Positivity}
\label{sec:proofmain}
In this Section, $\CA$, $X$, $\mu$, $\bI = (I_j\colon j\geq 1)$, \dots are as in 
Sections~\ref{subsec:Furstenberg} and~\ref{subsec:uniformity}. 
Given a sequence of intervals $(J_k\colon k\geq 1)$ in $\Z$ whose 
lengths tend to infinity, we assume that for each $j\geq 1$, there 
exists some $k=k(j)$ such that the interval $I_j$ is included in 
$J_k$.

We simplify the notation: we write $Z$ instead of $Z_d$, $\nu$ instead 
of $\mu_d$, $e$ instead of $e_d$. If $f$ is a function on $X$, 
$\widetilde f=\E(f\mid Z)$.

\subsection{Positivity}
\begin{lemma}
\label{lem:positive3}
Let $B\subset \Z$ be such that $\one_B\in\CA$ and let $f$ be the 
continuous function on $X$ associated to this set:
$$
 f(T^nx_0)=\one_B(n)\ .
$$
Let $m\geq 1$ be an integer and  let $h_\epsilon$, 
$\emptyset\neq\epsilon\subset[m]$, be $2^m-1$ 
nonnegative bounded measurable functions on $Z$. 
Assume that
$$
 \int\prod_{\emptyset\neq\epsilon\subset[m]}h_\epsilon(x_\epsilon)\,
d\nu_e\cube m(\bx)>0
$$
and that 
$$
 \Z\setminus B\text{ is not a }\PW\Nil_d\Bohr_0\ .
$$
Then
$$
 \int \widetilde f(x_{ \{m+1\} })\cdot
\prod_{\substack{\epsilon\subset[m+1]
\\ \epsilon\neq\emptyset,\{m+1\} }}
h_{\epsilon\cap[m] }(x_\epsilon)\, d\nu\cube{m+1}_e(\bx)>0\ .
$$
\end{lemma}

\begin{proof}
By Proposition~\ref{prop:positive2},
$$
\int\prod_{\substack{\epsilon\subset[m+1]\\ 
\epsilon\neq\emptyset,\{m+1\} }}
h_{\epsilon\cap[m] }(x_\epsilon)\,d\nu\cube{m+1}_\bbullet
(\bx)>0\ .
$$
For $z\in Z$, define
$$
 H(z)=\int\prod_{\substack{\epsilon\subset[m+1] \\ 
\epsilon\neq\emptyset,\ \epsilon\neq\{m+1\} }}
h_{\epsilon\cap[m] }(x_\epsilon)\,d\nu_{e, z}\cube{m+1}(\bx)\ .
$$
We have that $\delta:=H(e)>0$ and, by Proposition~\ref{prop:muax},
 $H$ is continuous on $Z_d$. Therefore, the subset
$$
 \Lambda=\{n\in\Z\colon H(T^ne)>\delta/2\}
$$
 is a $\Nil_d\Bohr$-set.

By the same proposition,
$$
 \int\widetilde f(x_{ \{m+1\} })\cdot
\prod_{\substack{\epsilon\subset[m+1] \\ 
\epsilon\neq\emptyset, \{m+1\} }}
h_{\epsilon\cap[m] }(x_\epsilon)\,d\nu_e\cube{m+1}
=\int \widetilde f(z)H(z)\,d\nu(z)\ .
$$
By Proposition~\ref{prop:intphif}, this last integral is equal to
$$
 \lim\Av_\bI f(T^nx_0) H(T^ne)\geq\frac{\delta}{2}
\limsup_j\frac 1{|I_j|}
|\Lambda\cap B\cap I_j|\ .
$$
If this $\limsup$ is equal to zero, then there exist arbitrarily long 
intervals $J_\ell$ such that $\Lambda\cap B\cap J_\ell=\emptyset$
and thus the set $\Z\setminus B$ contains $\Lambda\cap J_\ell$ for 
all $\ell$.  Therefore $\Z\setminus B$ is a $\PW\Nil_g\Bohr$-set, hence a contradiction.
\end{proof}

\begin{corollary}
\label{cor:positive3}
Let $B\subset \Z$ be such that $\one_B\in\CA$ and let $f$ be the 
continuous function on $X$ associated to this set. Assume that 
$\Z\setminus B$ is not a $\PW\Nil_d\Bohr_0$-set.
Then, for every $m$,
\begin{equation}
\label{eq:psitivem}
 \int\prod_{\emptyset\neq\epsilon\subset[m]} \widetilde f(x_\epsilon)\,
d\nu\cube m_e(\bx)>0\ .
\end{equation}
\end{corollary}

\begin{proof}
We remark first that $\int f\,d\mu>0$. Indeed, if this integral is 
zero, then the density of the set $B$ in the intervals $I_j$ 
converges to $0$ and $\Z\setminus B$ contains arbitrarily long 
intervals, a contradiction. 

We show~\eqref{eq:psitivem} by induction. 
We have that $\nu_{e }\cube 1=\delta_e\times\nu$ and thus
$$
 \int \widetilde f(x_1)\,d\nu_{d\ e}\cube 1(\bx)
=\int \widetilde f\,d\nu=\int  f\,d\mu>0\ .
$$
Assume that~\eqref{eq:psitivem} holds for some $m\geq 1$. Then
Lemma~\ref{lem:positive3} applied to $h=\widetilde f$ shows that it 
holds for $m+1$.
\end{proof}

\subsection{And now we gather all the pieces of the puzzle.}
Recall that if $E$ is a finite subset of $\N$, $\Sumset(E)$ is the set 
consisting in all sums of distinct elements of $E$ (the empty sum is 
not considered). A subset $A$ of 
$\Z$ is a $\Sumset_m^*$-set if $A\cap \Sumset(E)\neq\emptyset$ for every subset 
$E$ of $\N$ with $m$ elements.

We prove Theorem~\ref{th:sums}:
\begin{theorem*}
Let $A$ be a $\Sumset_{d+1}^*$ set. Then $A$ is a $\PW\Nil_d\Bohr$-set.
\end{theorem*}

\begin{proof}
Let $B=\Z\setminus A$, $\CA$  a subalgebra of $\ell^\infty(\Z)$ 
containing $\one_B$ and $X,\mu, \bI, \dots$ are as above.  The continuous 
function $f$ 
on $X$ is  associated to $\one_B$ and we use the same notation as above.

Assume that $A$ is not a $\PW\Nil_d\Bohr$-set.  
By Corollary~\ref{cor:positive3},
$$
  \int\prod_{\emptyset\neq\epsilon\subset[d+1]} \widetilde f(x_\epsilon)\,
d\nu\cube {d+1}_e(\bx)>0
$$
and by Proposition~\ref{prop:ouf}, this integral is equal to the 
iterated limit of the averages in $n=(n_1,\dots,n_{d+1})$ of
$$
 \prod_{\emptyset\neq\epsilon\subset[d+1]} f(T^{n\cdot\epsilon}x_0)
\ .
$$
This product is nonzero if and only if $\Sumset(\{n_1,\dots,n_{d+1}\})
\subset B$. But the complement $A$ of $B$ in 
$\Z$ is a $\Sumset_{d+1}^*$-set (recall that the $n_i$ belong to some 
of the intervals $I_j$), and so this can not 
happen.
\end{proof}

\section{Proof of Theorem~\ref{th:holes}}
\label{sec:iteration}
We now prove Theorem~\ref{th:holes} (recall that 
Theorem~\ref{th:Delta} is a particular case of this theorem):
\begin{theorem*}
Every $\SH_d^*$-set is a $\PW\Nil_d\Bohr_0$-set.
\end{theorem*}

\subsection{The method}
The proof is by contradiction. In this section, $d\geq 1$ is an 
integer and $A$
 is  a subset of the integers. We assume that $A$ is not 
a $\PW\Nil_d\Bohr_0$-set and by induction, we build an infinite sequence
$P=(p_j\colon j\geq 1)$ such that $A\cap\SH_d(P)=\emptyset$.

Let $\CA$ be a subalgebra of $\ell^\infty(\Z)$ 
containing $\one_A$ and let $X,\mu,\dots$ and the sequence of 
intervals $\bI=(I_j\colon j\geq 1)$ be as in 
Sections~\ref{subsec:Furstenberg} and~\ref{subsec:uniformity}.
We have the same conventions as in the preceding section for
the intervals $I_j$.

We write $B = \Z\setminus A$ and  let
 $f$ be the continuous function 
on $X$ associated to $\one_B$ (see Section~\ref{subsec:Furstenberg}):
$$
 f(T^nx_0)=\begin{cases} 1 & \text{if } n\in B\ ;\\
0 &\text{otherwise.}
\end{cases}
$$
As in Section~\ref{sec:proofmain}, we simplify the notation: we write $Z$ instead of $Z_d$, $\nu$ instead 
of $\mu_d$, and $e$ instead of $e_d$. If $f$ is a function on $X$, 
$\widetilde f=\E(f\mid Z)$.  

In this section, it is more convenient to index points of 
$X\type d$ by $\{0,1\}^d$ instead of by $\CP([d])$.  Thus a 
point $\bx\in X\type d$ is written $\bx = (x_\epsilon\colon 
\epsilon\in\{0,1\}^d)$.

For every $j\geq 1$, by induction we build $2^d-1$ continuous 
nonnegative functions $h_{\epsilon}^{(j)}$, 
$00\ldots 0\neq\epsilon\in\{0,1\}^d$, on $X$ satisfying
\begin{equation}
\label{eq:recurr}
 \int\prod_{\substack{\epsilon\in\{0,1\}^d\\ 
\epsilon\neq 00\ldots 0}}\widetilde 
h_{\epsilon}^{(j)}(x_\epsilon)\,d\nu\cube d_e(\bx)>0\ .
\end{equation}

We start by setting all of the
functions $h_{\epsilon}^{(0)}$, $00\ldots 0\neq\epsilon\in\{0,1\}^d$, to be  
equal to $f$.
By Corollary~\ref{cor:positive3} applied with $m=d$ and 
rewritten in the 
current notation, we have that property~\eqref{eq:recurr} is 
satisfied for $i=0$.  

\subsection{Iteration}
Assume $j\geq 1$ and that 
property~\eqref{eq:recurr} is satisfied for $j-1$.

By Proposition~\ref{prop:positive2},
$$
 \int\prod_{\substack{\epsilon\in\{0,1\}^{d+1}\\ 
\epsilon\neq 00\dots 0, \epsilon\neq 00\dots 01 }}
\widetilde h_{\epsilon_1\dots\epsilon_d}^{(j-1)}(x_\epsilon)
\,d\nu\cube {d+1}_\bbullet(\bx)>0 \ .
$$
By Lemma~\ref{lem:positive3}, rewritten in our current notation, we 
have that 
 $$
 \int\widetilde f(x_{00\dots 01})\cdot
\prod_{\substack{\epsilon\in\{0,1\}^{d+1}\\ \epsilon\neq 00\dots 0, 
\epsilon\neq 00\dots 01 }}
\widetilde h_{\epsilon_1\dots\epsilon_d}^{(j-1)}(x_\epsilon)
\,d\nu\cube {d+1}_e(\bx)>0\ .
$$
For convenience, we write $h_{00\dots0}^{(j-1)} =f$ and rewrite this 
equation as
\begin{equation}
\label{eq:horreur2}
 \int
\prod_{\substack{\epsilon\in\{0,1\}^{d+1}\\ \epsilon\neq 00\dots 0}}
\widetilde h_{\epsilon_1\dots\epsilon_d}^{(j-1)}(x_\epsilon)
\,d\nu\cube {d+1}_e(\bx)>0\ .
\end{equation}
By Proposition~\ref{prop:ouf}, this last integral is the iterated 
limit of the averages for $n=(n_1,\dots,n_{d+1})$ of 
$$
 \prod_{\substack{\epsilon\in\{0,1\}^{d+1}\\ \epsilon\neq 00\dots 0}}
h_{\epsilon_1\dots\epsilon_d}^{(j-1)}(T^{n\cdot\epsilon}x_0)\ .
$$
We make a change of indices, writing elements of $\Z^{d+1}$ as 
$(p,n_1,\dots,n_d)$ and setting $n=(n_1,\dots,n_d)$.
Elements of $\{0,1\}^{d+1}$ are written as 
$\eta\epsilon_1\dots\epsilon_d$ with $\eta\in\{0,1\}$ and we set
$\epsilon=\epsilon_1\dots\epsilon_d$.
The last product becomes:
$$
 h_{100\dots 0}^{(j-1)}(T^px_0)
\prod_{\substack{\epsilon\in\{0,1\}^d \\
\epsilon\neq 00\dots 0}}
\bigl(h_{0\epsilon_1\dots\epsilon_{d-1}}^{(j-1)}\cdot 
T^ph_{1\epsilon_1\dots\epsilon_{d-1} }^{(j-1)}\bigr)
(T^{ n\cdot\epsilon} x_0)\ .
$$

For $\epsilon\in\{0,1\}^d$, $\epsilon\neq 00\dots 0$, and for $p\in \Z$, 
set
$$
 g_{p,\epsilon}=h_{0\epsilon_1\dots\epsilon_{d-1}}^{(j-1)}\cdot
T^p h_{1\epsilon_1\dots\epsilon_{d-1}}^{(j-1)}
$$
and rewrite the last expression as
$$
 h_{100\dots 0}^{(j-1)}(T^px_0)\prod_{\substack{
\epsilon\in\{0,1\}^d \\ \epsilon\neq 00\dots 0}}
g_{p,\epsilon}(T^{n\cdot\epsilon}x_0)\ .
$$
By Proposition~\ref{prop:ouf} again, the iterated limit of the 
averages in $n_1,\dots,n_{d}$ converges to
\begin{multline*}
  h_{100\dots 0}^{(j-1)}(T^px_0)\int
\prod_{\substack{
\epsilon\in\{0,1\}^d \\ \epsilon\neq 00\dots 0}}
\E\bigl(g_{p,\epsilon}\mid Z_{d-1}\bigr)
(x_\epsilon)\, d\mu\cube d_{d-1\ e}(\bx)
\\
 =  h_{100\dots 0}^{(j-1)}(T^px_0)\int
\prod_{\substack{
\epsilon\in\{0,1\}^d \\ \epsilon\neq 00\dots 0}}
\widetilde{g_{p,\epsilon}}(x_\epsilon)\, d\mu\cube d_e(\bx)
\end{multline*}
because the measure  $\mu\cube d_e$ is relatively independent with 
respect to $\mu\cube m_{d-1\ e}$ (see 
Proposition~\ref{prop:mesuresmua}, part (d)).

The averages in $p$ over the intervals $\bI$ of this expression 
converge to the limit~\eqref{eq:horreur2}, which is positive.  Thus 
there exists some $p$ (belonging to some $I_i$) such that this 
expression is positive. Choosing $p_{j}$ to be 
this $p$, for $00\ldots 
0\neq\epsilon\in\{0,1\}^d$, we define
$$
 h_\epsilon^{(j)}=g_{p_{j},\epsilon}=
h_{0\epsilon_1\dots\epsilon_{d-1}}^{(j)}\cdot
T^{p_{j}} h_{1\epsilon_1\dots\epsilon_{d-1}}^{(j-1)}
$$
(recall that $h_{00\dots 0}^{(j-1)}=f$).  
Since~\eqref{eq:recurr} is valid with 
$h_\epsilon^{(j)}$ substituted for $h_\epsilon^{(j-1)}$, we can iterate.
Moreover, 
$$
h_{100\dots 0}^{(j-1)}(T^{p_{j}}x_0)>0 \ .
$$

\subsubsection{Interpreting the iteration}
By induction, it follows that for every $j\geq 0$, 
the functions $h_{j,\epsilon}$, $00\ldots 0\neq\epsilon\in\{0,1\}^d$, 
only depend on the 
first nonzero digit of $\epsilon$:
$$
 h_{j,\epsilon}=\phi_{j,k}\text{ if }
\epsilon_1=\dots=\epsilon_{k=1}=0\text{ and }\epsilon_k=1\ .
$$
We have the inductive relations
\begin{gather*}
\phi_{0,k}=f\text{ for } 1\leq k\leq d\ ; \\
\phi_{j-1,1}(T^{p_{j}}x_0)>0\ ; \\
\text{for }1\leq k<d,\quad 
\phi_{j,k}=\phi_{j-1,k+1}\cdot T^{p_{j}}\phi_{j-1,1}\ ;\\
\phi_{j,d}=f\cdot T^{p_{j}}\phi_{j-1,1}\ .
\end{gather*}
By induction, $\phi_{j,1}\leq\phi_{j,2}\leq\dots\leq\phi_{j,d}\leq f$.
Moreover, we deduce the following relations between 
the functions $\phi_{j,1}$:
\begin{align*}
 & \text{for }1\leq j< d, \ 
\phi_{j,1} = f \cdot \prod_{k=1}^j
T^{p_{j-k+1}}\phi_{j-k, 1} \\
& \text{for }j\geq d, \ 
\phi_{j,1} = f \cdot \prod_{k=1}^{d}
T^{p_{j-k+1}}\phi_{j-k, 1}  \ .
\end{align*}

For every $j$, there is a finite  set $E_j$ of integers with
$$
\phi_{j,1}=\prod_{q\in E_j}T^qf\ .
$$
We have that $E_0 = \{0\}$ and the $E_j$ satisfy that relations
\begin{gather*}
\text{for }1\leq j< d,\ 
E_j = \{0\} \cup (E_{j-1}+p_j)\cup (E_{j-2}+p_{j-1})
\cup\ldots \cup (E_{0}+p_{1})\ ,
\\
\text{for }j\geq d,\ 
E_j = \{0\}\cup (E_{j-1}+p_j)\cup (E_{j-2}+p_{j-1})
\cup\ldots \cup (E_{j-d}+p_{j-d+1})\ .
\end{gather*}
By induction, $E_j$ consists in all sums of the form 
$\epsilon_1p_1+\dots+\epsilon_jp_j$ 
where
 $\epsilon_i\in\{0,1\}$ for all $i$,
and, after the first occurrence of $1$, there can be no block of $d$ 
consecutive $0$'s.

By induction, each function $\phi_{j,1}$ only takes on the values 
of $0$ and $1$ and  corresponds to a subset $B_j$ 
of the integers and we have
$$
 B_j=\bigcap_{q\in E_j}(B-q)\ .
$$
For every $j$, since  $\phi_{j-1,1}(T^{p_{j}}x_0)>0$,
we have that $p_{j}\in B_{j-1}$ and thus that $E_{j-1}+p_{j}\subset B$.

We conclude 
that all sums of the form $\epsilon_1p_1+\dots+\epsilon_kp_k$ 
with $\epsilon_i\in\{0,1\}$ for all $i$ belong to $B$, 
provided the $\epsilon_i$ are not all equal to $0$ and that the 
blocks of consecutive $0$'s between two $1$'s have length $< d$.
In other words, $B\supset \SH_d(\{p_j\colon j\geq 1\})$ and we have a 
contradiction.
\qed

\bigskip

We note that at each step in the iteration, we have infinitely many 
choices for the next $p$.  In particular, we can take the $p_j$ 
tending to infinity as fast as we want. 
More interesting, 
in the construction we can choose a different permutation of 
coordinates at each step.  This gives rise to different, but related, 
structures, which do not seem to have any simple description.

\end{document}